\documentclass[final,onefignum,onetabnum]{siamart190516}

\usepackage{braket,amsfonts}

\usepackage{array}
\usepackage{bm}
\usepackage{mathrsfs}

\usepackage[caption=false]{subfig}
\usepackage{pgfplots}

\newsiamthm{claim}{Claim}
\newsiamremark{remark}{Remark}
\newsiamremark{hypothesis}{Hypothesis}
\crefname{hypothesis}{Hypothesis}{Hypotheses}

\usepackage[linesnumbered,ruled]{algorithm2e}

\usepackage{graphicx,epstopdf}

\Crefname{ALC@unique}{Line}{Lines}

\usepackage{amsopn}

\usepackage{xspace}
\usepackage{bold-extra}
\usepackage[most]{tcolorbox}

\colorlet{texcscolor}{blue!50!black}
\colorlet{texemcolor}{red!70!black}
\colorlet{texpreamble}{red!70!black}
\colorlet{codebackground}{black!25!white!25}

\lstdefinestyle{siamlatex}{%
  style=tcblatex,
  texcsstyle=*\color{texcscolor},
  texcsstyle=[2]\color{texemcolor},
  keywordstyle=[2]\color{texemcolor},
  moretexcs={cref,Cref,maketitle,mathcal,text,headers,email,url},
}

\tcbset{%
  colframe=black!75!white!75,
  coltitle=white,
  colback=codebackground, 
  colbacklower=white, 
  fonttitle=\bfseries,
  arc=0pt,outer arc=0pt,
  top=1pt,bottom=1pt,left=1mm,right=1mm,middle=1mm,boxsep=1mm,
  leftrule=0.3mm,rightrule=0.3mm,toprule=0.3mm,bottomrule=0.3mm,
  listing options={style=siamlatex}
}

\newtcblisting[use counter=example]{example}[2][]{%
  title={Example~\thetcbcounter: #2},#1}

\newtcbinputlisting[use counter=example]{\examplefile}[3][]{%
  title={Example~\thetcbcounter: #2},listing file={#3},#1}

\DeclareTotalTCBox{\code}{ v O{} }
{ 
  fontupper=\ttfamily\color{black},
  nobeforeafter,
  tcbox raise base,
  colback=codebackground,colframe=white,
  top=0pt,bottom=0pt,left=0mm,right=0mm,
  leftrule=0pt,rightrule=0pt,toprule=0mm,bottomrule=0mm,
  boxsep=0.5mm,
  #2}{#1}

\patchcmd\newpage{\vfil}{}{}{}
\usepackage{extramath}
\usepackage{mathtools}

\def\beq{\begin{equation}}
\def\eeq{\end{equation}}

\def\bgam{ {\boldsymbol{\gamma}} }

\numberwithin{equation}{section}

\def\spn{\mathrm{span}}
\newcommand{\spnorm}[1]{{\|{#1}\|_{\mathrm{sp}}} }

\def\beq{\begin{equation}}
\def\eeq{\end{equation}}

\def\cN{ {{\mathcal N}}}

\def\cP{ {{\mathcal P}}}

\def\fQ{ {{ \mathfrak Q}}}
\def\fZ{ {{ \mathfrak Z}}}

\def\mathand{{\quad \quad \mathrm{and} \quad \quad}}

\def\R{ {\mathbb{R}} }
\def\K{ {\mathbb{K}} }
\def\C{ {\mathbb{C}} }

\def\N{ {\mathbb{N}} }

\def\cpd{{\cp{\bA,\bB,\bC}}}
\def\cA{ {\mathcal A} }

\def\cM{{\mathcal M}}
\def\cN{ {\mathcal N} }

\def\cP{{\mathcal P}}

\def\cS{{\mathcal S} }

\def\cT{{\mathcal T}}

\def\cY{{\mathcal Y}}

\def\bA{ {\mathbf{A}} }
\def\bB{ {\mathbf{B}} }
\def\bC{ {\mathbf{C}} }

\def\hbC{ {\hat{\bC}} }

\def\bE{ {\mathbf{E}} }

\def\bM{ {\mathbf{M}} }
\def\bfm{ {\mathbf{m}} }

\def\bP{ {\mathbf{P}} }

\def\bS{ {\mathbf{S}} }

\def\bT{ {\mathbf{T}} }

\def\hbT{ {\hat{\mathbf{T}}} }
\def\bU{ {\mathbf{U}} }
\def\bV{ {\mathbf{V}} }

\def\bX{ {\mathbf{X}} }
\def\bY{ {\mathbf{Y}} }
\def\bZ{ {\mathbf{Z}} }
\def\bQ{ {\mathbf{Q}} }
\def\bS{ {\mathbf{S}} }
\def\bR{ {\mathbf{R}} }

\def\teT{ {\mathrm{T}}  }
\def\teH{ {\text{H}}  }

\def\ba{ {\vec{a}} }
\def\bb{ {\vec{b}} }
\def\bc{ {\vec{c}} }

\def\hbc{ {\hat{\bc}} }
\def\bq{ {\vec{q}} }

\def\bv{ {\vec{v}} }

\def\bx{ {\vec{x}} }
\def\by{ {\vec{y}} }
\def\bz{ {\vec{z}} }

\def\lam{ {\lambda} }

\def\blam{ {\vec{\lambda}} }
\def\bLam{ { {\mathfrak{L}}} }
\def\tblam{ {\tilde{\vec{\lambda}}} }

\def\hbLam{ {\hat{\bLam}} }

\def\b0{ {\vec{0}} }

\def\hcT{{\hat{\cT}}}

\def\bfballK{{  \mathrm{BF}^{\K,0}_{\epsilon}  }}

\def\bfballNK{{  \mathrm{BF}^{\K,N}_{\epsilon}  }}
\def\bfballNplK{{  \mathrm{BF}^{\K,N+1}_{\epsilon}  }}
\def\bfregK{{  \mathrm{BF}^{\K}_{\epsilon}  }}

\def\bfballC{{  \mathrm{BF}^{\C,0}_{\epsilon}  }}

\def\bfballNC{{  \mathrm{BF}^{\C,N}_{\epsilon}  }}

\def\bfregC{{  \mathrm{BF}^{\C}_{\epsilon}  }}
\def\bfballR{{  \mathrm{BF}^{\R,0}_{\epsilon}  }}

\def\bfballNR{{  \mathrm{BF}^{\R,N}_{\epsilon}  }}

\def\bfregR{{  \mathrm{BF}^{\R}_{\epsilon}  }}

\def\Gr{{ \mathrm{Gr} }}

\def\boldeta{ {\boldsymbol{\eta}} }

\def\bzeta{ {\boldsymbol{\zeta}} }

\def\scriptsize{}
\def\rank{\mbox{rank}}

\def\cS{{\mathcal S}}

\newcommand{\df}[1]{{\bf{#1}}{\index{#1}}}

\def\mat{\partial^{\mathrm{mat}}}

\def\ran{\mathrm{ran \ }}

\def\be{ {\mathbf{e}} }

\begin{tcbverbatimwrite}{tmp_\jobname_header.tex}
\title{A recursive eigenspace computation for the Canonical Polyadic decomposition\thanks{\funding{Research supported by: (1) Flemish Government: This work was supported by the Fonds de la Recherche Scientifique--FNRS and the Fonds Wetenschappelijk Onderzoek--Vlaanderen under EOS Project no 30468160 (SeLMA); (2) KU Leuven Internal Funds C16/15/059 and ID-N project no 3E190402. (3) This research received funding from the Flemish Government under the “Onderzoeksprogramma Artificiële Intelligentie (AI) Vlaanderen” programme.  (4) Work supported by Leuven Institute for Artificial Intelligence (Leuven.ai). (5) In addition, Michiel Vandecappelle was partially supported by an SB Grant (1S66617N) from the Research Foundation---Flanders (FWO).}}}

\author{Eric Evert\footnotemark[2] 
\and Michiel Vandecappelle\footnotemark[2]
\and Lieven De Lathauwer\thanks{KU Leuven, Dept. of Electrical Engineering ESAT/STADIUS, Kasteelpark
Arenberg 10, bus 2446, B-3001 Leuven, Belgium and KU Leuven – Kulak, Group Science, Engineering and Technology, E.
Sabbelaan 53, B-8500 Kortrijk, Belgium
\newline \hbox{\quad} (Eric.Evert, Michiel.Vandecappelle, Lieven.DeLathauwer)@kuleuven.be.}}

\headers{ Recursive eigenspace decomposition for CPD}{Eric Evert, Michiel Vandecappelle, and Lieven De Lathauwer}
\end{tcbverbatimwrite}
\input{tmp_\jobname_header.tex}

\begin{document}
\maketitle

\begin{tcbverbatimwrite}{tmp_\jobname_abstract.tex}
\begin{abstract}
The canonical polyadic decomposition (CPD) is a compact decomposition which expresses a tensor as a sum of its rank-1 components. 
A common step in the computation of a CPD is computing a generalized eigenvalue decomposition (GEVD) of the tensor.
A GEVD provides an algebraic approximation of the CPD which can then be used as an initialization in optimization routines.

While in the noiseless setting GEVD exactly recovers the CPD, it has recently been shown that pencil-based computations such as GEVD are not stable. In this article we present an algebraic method for approximation of a CPD which greatly improves on the accuracy of GEVD. Our method is still fundamentally pencil-based; however, rather than using a single pencil and computing all of its generalized eigenvectors, we use many different pencils and in each pencil compute generalized eigenspaces corresponding to sufficiently well-separated generalized eigenvalues. The resulting ``generalized eigenspace decomposition" is significantly more robust to noise than the classical GEVD. 

Accuracy of the generalized eigenspace decomposition is examined both empirically and theoretically. In particular, we provide a deterministic  perturbation theoretic bound which is predictive of error in the computed factorization. 
\end{abstract}


\end{tcbverbatimwrite}
\input{tmp_\jobname_abstract.tex}


\section{Introduction}

Tensors, or multiindexed arrays, play an important role in fields such as machine learning and signal processing \cite{Cetal15,Setal17}. These higher-order generalizations of matrices allow for preservation of higher-order structure present in data and low-rank decompositions of tensors allow for compression of data and recovery of underlying information \cite{Setal17,Cetal15,CJ10}. One of the most popular decompositions for tensors is the canonical polyadic decomposition (CPD) which expresses a tensor as a sum of rank one tensors. An important feature of the CPD is that, with mild assumptions \cite{DD13,K77,SB00}, the CPD of a low-rank tensor is unique. Furthermore, for a low-rank tensor with a unique CPD, the CPD can often be found algebraically. Such an algebraic solution can typically be obtained with limited computation time, hence is often used as an initialization for optimization based methods when the tensor is noisy.

One of the most popular algorithms for algebraic computation of a CPD is the generalized eigenvalue decomposition (GEVD) \cite{DD14,DD17,SK90,LRA93}. The key idea behind the classical GEVD is that a factor matrix of a tensor may be obtained by computing the generalized eigenvectors of any subpencil of the tensor. However, it has recently been shown that pencil-based algorithms such as GEVD are not stable \cite{BBV19}. That is, the condition number for computing a generalized eigenvalue decomposition of a subpencil can be arbitrarily larger than the condition number \cite{ BV18} for computing the CPD of the tensor. In this article we present an extension of the GEVD algorithm which significantly improves the accuracy of algebraic computation of the CPD. This increase in accuracy is obtained by using many subpencils of the tensor to compute the CPD. In each subpencil we only compute generalized eigenvectors and eigenspaces which correspond to sufficiently well-separated generalized eigenvalues.

The ability to accurately compute generalized eigenvectors of a matrix pencil is heavily dependent on the separation between the generalized eigenvalues and the generalized eigenvectors of the pencil \cite{SS90,GV96,N07}. In the case that the generalized eigenvalues and eigenvectors are well-separated, a GEVD may be accurately computed. However, when either a pair of generalized eigenvalues or generalized eigenvectors is near parallel, the accuracy of the computed generalized eigenvectors deteriorates. As such, the GEVD algorithm performs well when there is a subpencil for the tensor in which all generalized eigenvalues are well-separated; however, GEVD runs into challenges if one is unable to find a subpencil in which all generalized eigenvalues are well-separated. 

When computing a GEVD of a rank $R$ tensor, one computes the generalized eigenvalues and eigenvectors of a pair of $R \times R$ matrices.  The generalized eigenvalues of this matrix pair can be naturally interpreted as a collection of $R$ points on the unit circle \cite{SS90}, so it is not difficult to see that the generalized eigenvalues cannot be well-separated when the tensor rank $R$ is sufficiently large.

We propose an alternative approach in which we exploit the fact that a tensor has many subpencils, and that a factor matrix of the tensor may be computed by combining information from many different subpencils. Using the language of the joint generalized eigenvectors (JGE vectors) \cite{EvertD19}, the primary goal of the GEVD algorithm is in fact to compute the JGE vectors of the tensor.

Let $\K=\R$ or $\C$. Suppose the tensor $\cT \in \K^{R \times R \times K}$ is a rank $R$ tensor with JGE vectors $\{\bx_r\}_{r=1}^R$. Then for every subpencil of $\cT$, the generalized eigenspaces are of the form $\spn (\{\bx_{r_1}, \dots, \bx_{r_k}\})$ for some subcollection of JGE vectors of $\cT$. Here $\{r_1,\dots,r_k\}$ is a subcollection of indices from $\{1,\dots,R\}$ with cardinality $k$.  Intuitively it follows that one may compute the JGE vectors of $\cT$ by computing  generalized eigenspaces for many different subpencils of $\cT$. 

An important feature of this approach is that the accuracy of computing a particular generalized eigenspace does not depend on the separation of the generalized eigenvalues which correspond to that  generalized eigenspace. Rather the accuracy depends on the separation of the set of generalized eigenvalues corresonding to the generalized eigenspace of interest from the other generalized eigenvalues of the pencil as well as the separation between the generalized eigenspace itself and the other generalized eigenvectors of the pencil \cite{S73}. Thus by only computing generalized eigenspaces  which  correspond to collections of generalized eigenvalues that are well-separated from the other generalized eigenvalues of the pencil, one source of inaccuracy for the generalized eigenvalue decomposition, i.e. inaccuracy introduced by projecting a tensor to a matrix pencil \cite{BBV19}, is significantly reduced. 

\subsection{Notation and definitions}

Let $\K$ denote either $\R$ or $\C$. We denote scalars, vectors, matrices, and tensors by lower case $(a)$, bold lower case $(\vec{a})$, bold upper case $(\bA)$, and calligraphic script $(\cA)$, respectively. Additionally, fraktur font $(\mathfrak{A})$ is used to denote a subspace of $\K^I$. For a matrix $\bA \in \K^{I \times I}$, we let $\bA^{\teT}$ denote the transpose of $\bA$ while $\bA^\teH$ denotes the conjugate transpose of $\bA$. If $\bA$ is invertible, then we let $\bA^{-\teT}$ denote the inverse of $\bA^\teT$. A \df{tensor} is a multiindexed array with entries in $\K$. In this article we restrict our attention to tensors of order three. In other words, the tensors we consider are elements of $\K^{I_1 \times I_2 \times I_3}$. 

Given nonzero vectors $\ba \in \K^{I_1},\bb \in \K^{I_2}, \bc \in \K^{I_3}$, let 
\[
\ba \op \bb \op \bc \in \K^{I_1 \times I_2 \times I_3}
\]
denote the $I_1 \times I_2 \times I_3$ tensor with $i,j,k$ entry equal to $a_i b_j c_k$. A tensor of this form is called a \df{rank one tensor}. The minimal integer $R$ such that $\cM \in \K^{I_1 \times I_2 \times I_3}$ may be written as
\[
\cM= \sum_{r=1}^R \ba_r \op \bb_r \op \bc_r
\]
where the vectors $\ba_r,\bb_r,\bc_r$ have entries in $\K$ is called the $\K$-rank of the tensor $\cM$, and a decomposition of this form is called a {canonical polyadic decomposition} (CPD) of $\cM$. Compactly we write
$
\cM=\cpd  
$
where the matrices $\bA,\bB$ and $\bC$ have $\ba_r,\bb_r$ and $\bc_r$ as their $r$th columns, respectively. The matrices $\bA,\bB,\bC$ are called the \df{factor matrices} for $\cM$. 

Given a tensor $\cM \in \K^{I_1 \times I_2 \times I_3}$, let  $ \bM_{[1;3,2]}$   denote an $I_1 \times I_2 I_3$ matrix with columns $\{\bfm(:,k,\ell)\}_{\forall k,\ell}$. Here, the columns of $\bM_{[1;3,2]}$ are ordered so that the second index of $\bfm(:,k,\ell)$ increments faster than the third. The matrices $ \bM_{[2;3,1]}$ and $\bM_{[3;2,1]}$ are defined similarly. A matrix of the form $\bM_{[i;k,j]}$ is called a \df{mode-$i$ unfolding} of $\cM$.

Having defined the mode-$i$ unfoldings of $\cM$, we define 
\begin{align}
R_i (\cM) &\coloneqq \rank (\bM_{[i;k,j]}) \qquad \mathrm{for \ } i=1,2,3.
\end{align}
The integer triple $(R_1 (\cM), R_2 (\cM), R_3 (\cM))$ is called the \df{multilinear rank} of $\cM$. 

For tensors $\cM$ occurring in applications, $\cM$ is often thought of as a noisy measurement of some signal tensor of interest. That is, one can decompose $\cM = \cT' +\cN'$ where $\cT'$ is the true signal tensor of interest and $\cN'$ is measurement error. Most commonly the signal tensor $\cT'$ has low multilinear rank. That is, one has $R_i (\cT')  <\!\!< I_i$ for each $i$. In this case, under an appropriate unitary change of bases one has  $t'_{i,j,k} = 0$ if $i>R_1 (\cT')$ or $j>R_2 ( \cT')$ or $k>R_3 (\cT')$. In this choice of bases it is clear that $\cT'$ may be compressed to a tensor of size $R_1 (\cT') \times R_2 (\cT') \times R_3 (\cT')$ without loss of information. Such a compression is called an \df{orthogonal compression} of $ \cT'$ to its multilinear rank, and a popular choice of orthogonal compression is the multilinear singular value decomposition (MLSVD). See \cite{DDV00b} for details. 

Compression of a noisy measurement $\cM=\cT'+\cN'$ to the multilinear rank of $\cT'$ is a common step in computation of a CPD which both approximates the (core of) the signal tensor $\cT'$ and reduces computational complexity. However, such a compression has very few implications for the methods discussed in this article. For this reason we often restrict to considering tensors $\cY$ with the property $R_i(\cY)=I_i$ for $i=1,2,3$. For clarity, here $I_i$ refers to the $i^{th}$ dimension of the compressed tensor $\cY$. Intuitively, the symbol $\cY$ is used to denote a tensor which can be thought of as a compression of $\cM=\cT'+\cN'$ to the multilinear rank of $\cT'$. In particular, we may write $\cY = \cT+\cN$ where $\cT$ is an orthogonal compression $\cT'$ to its multilinear rank and $\cN$ is noise.

It is not difficult to show that the $\K$-rank of a tensor $\cY$ is bounded from below by each of the $R_i(\cY)$,  e.g., see \cite{Lan12}. For most tensors the inequality is strict. However, for the signal portion $\cT$ of a tensor $\cY$ occurring in applications, one often has (up to a permutation of indices) $\rank_\K (\cT)=R_1(\cT)=R_2(\cT) \geq R_3(\cT)$. In light of this, the tensors $\cY  = \cT +\cN$ we consider are taken to have size $R \times R \times K$ where $K \leq R$ and the signal portion $\cT$ of $\cY$ is assumed to have $\K$-rank $R$.\footnote{Note that one always has $\rank_\K (\cT') \geq R_3(\cT')$ hence the case $K>R$ does not occur after  orthogonal compression.} To summarize, throughout the article we let $\cM$ denote an arbitrary tensor of size $I_1 \times I_2 \times I_3$, we let $\cY$ denote a full multilinear rank tensor of size $R \times R \times K$, and we let $\cT$ denote a $R \times R \times K$ tensor with $\K$-rank equal to $R$ and full multilinear rank.

Define the \df{mode-1 product} between a tensor $\cM \in \K^{I_1 \times I_2 \times I_3}$ and a matrix $\bU \in \K^{J_1 \times I_1}$ to be the tensor $\cM \cdot_1 \bU \in \K^{J_1 \times I_2 \times I_3}$ with mode-$1$ fibers given by
\[
( \cM \cdot_1 \bU)(:,k,\ell) = \bU \cdot \vec{m} (:,k,\ell) \qquad \mathrm{for \ all \ } k,\ell. 
\]
The mode-$2$ and mode-$3$ products are defined analogously. Say a tensor $ \cY \in \K^{R \times R \times K}$ is \df{slice mix invertible} if there is a vector $\bv \in \R^{K}$ such that the matrix $\cY \cdot_3 \bv^{\text{T}}$ is invertible. That is, $ \cY$ is slice mix invertible if there exists a linear combination of the matrices $\bY(:,:,k)$ for $k=1,\dots, K$ which is invertible.  We call the matrices $\bY(:,:,k)$ the \df{frontal slices} of $\cY$ and use $\bY_k$ to denote $\bY(:,:,k)$ for $k=1,\dots, K$.

Let $\cY \in \K^{R \times R \times K}$. An \df{implicit subpencil} of $ \cY$ is a tensor of the form $ \cY \cdot_3 \bV \in \K^{R \times R \times 2}$ where $\bV \in \K^{2 \times K}$ has rank $2$. More generally, call a tensor of the form $ \cY \cdot_3 \bV \in \K^{R \times R \times L}$ where $\bV \in \K^{L \times K}$ has full row rank an \df{implicit subtensor} of $ \cY$. For emphasis we note that implicit subpencils and implicit subtensors differ from the usual notions in that they can be revealed by modal multiplication in the third mode of the tensor. Throughout the article, implicit subtensors and subpencils of a given tensor are denoted using the calligraphic letter $\cS$. In the reminder of the article we drop the distinction ``implicit", as all subpencils and subtensors we consider are implicit.

\subsection{Joint generalized eigenvalues and eigenvectors}
Given a tensor $\cT \in \K^{R \times R \times K}$, we say a nonzero vector $ \bz \in \K^R$ is a \df{joint generalized eigenvector} (JGE vector) of $ \cT$ if there exists a $\blam \in \K^K$ and a nonzero vector $ \bq \in \K^R$ such that
\beq
\label{eq:JGEVdef}
\bT_{k}  \bz = \lam_{k}  \bq
\eeq
for all $k = 1, \dots, K$. Note that assuming $ \bq \neq \b0$ guarantees that $\lam_k=0$ if and only if $ \bT_k \bz =\b0$. If $(\blam,\bz)$ is a solution to \eqref{eq:JGEVdef}, then the one-dimensional subspace $\bLam \coloneqq \spn (\blam)  \subset \K^K$ is a \df{joint generalized eigenvalue} (JGE value) of $ \cT$ and $(\bLam,\bz)$ is a \df{joint generalized eigenpair} (JGE pair) of $\cT$.\footnote{If $K=2$, then joint generalized eigenvectors and joint generalized eigenvalues are exactly equal to classical generalized eigenvalues and generalized eigenvectors (in the sense of \cite{SS90}).} Throughout the article we make frequent use of the following theorem \cite[Theorem 2.3 (1)]{EvertD19} which explains a connection between JGE vectors and tensor rank.

\begin{theorem}
\label{theorem:rankvsmult}
Let $\cT' \in \K^{I_1 \times I_2 \times I_3}$ and assume $\cT'$ has multilinear rank $(R,R,K)$. Let $\cT \in \K^{R \times R \times K}$ be an orthogonal compression of $\cT'$. Then $\cT'$ has $\K$-rank $R$ if and only if the JGE vectors of $\cT$ span $\K^R$.  Furthermore, if $\cT=\cpd$ has $\K$-rank $R$, then the JGE values of $\cT$ are equal to the spans of columns of the factor matrix $\bC$, and the columns of $\bB^{-\mathrm{T}}$ are JGE vectors of $\cT$. 
\end{theorem}

Given a subspace $\fZ \subset \K^R$ of dimension $\ell$ and a tensor $\cT \in \K^{R \times R \times K}$, say $\fZ$ is a \df{joint generalized eigenspace} of $\cT$ if there is a subspace $\fQ \subset \K^R$ of dimension less than or equal to $\ell$ such that
\[
\bT_k \fZ \subseteq \fQ \qquad \mathrm{ \ for \ all \ } k=1,\dots, K. 
\]
 Here $\bT_k \fZ$ denotes the subspace $\{\bT_k \bz : \bz \in \fZ\} \subset \K^R$. In the case that $K=2$ we simply call $\fZ$ a \df{generalized eigenspace}.\footnote{Generalized eigenspaces are sometimes called deflating subspaces by other authors.} Given a JGE value $\bLam \subset \K^{K}$ of a tensor $\cT \in \K^{R \times R \times K}$, let $\fZ_\bLam \subset \K^R$ be the subspace of joint generalized eigenvectors corresponding to $\bLam$. One may easily show that $\fZ_\bLam$ is a joint generalized eigenspace of $\cT$.

\subsubsection{Generalized eigenspaces of subtensors}

The main goal of this article is to improve the accuracy of algebraic (pencil-based) methods for computing a CPD. Key observations which allow us to accomplish this are that generalized subspaces may often be more stably computed than individual generalized eigenvectors and that the joint generalized eigenvectors of a tensor may be recovered by computing generalized eigenspaces for sufficiently many subtensors. The following proposition helps formalize this idea.  
\begin{proposition}
\label{proposition:FormOfEigenspaces}
Let $\cT \in \K^{R \times R \times K}$ be a tensor of $\K$-rank $R$ and let $\bV \in \K^{K \times L}$ be a matrix with full row rank. Set $\cS= \cT \cdot_3 \bV^{\teT}$, assume $\cS$ is slice mix invertible, and let $\bLam_\cS$ be a JGE value of $\cS$. There exist JGE values $\bLam_1, \dots, \bLam_J$ of $\cT$ such that
\[
\fZ_{\bLam_\cS} = \spn (\{ \mathfrak{Z}_{\mathfrak{L}_j}\}_{j=1}^J).
\]
\end{proposition}
\begin{proof}
First note that assuming $\cS$ is slice mix invertible implies $\cT$ is slice mix invertible. Now, following the proof of \cite[Proposition 5.1]{EvertD19}, one finds that each joint generalized eigenvalue of $\cS$ has the form $\bV^{\teT} \bLam$ for some JGE value $\bLam$ of $\cT$. Note that one must have $\bV^\teT \bLam \neq \b0$ as otherwise Theorem \ref{theorem:rankvsmult} would imply that $\cS$ has rank less than $R$, contradicting the assumption that $\cS$ is slice mix invertible. Given some JGE value $\bLam_\cS$ of $\cS$, let $\bLam_1,\dots,\bLam_J$ be those joint generalized eigenvalues of $\cT$ such that $\bV^{\teT} \bLam_j=\bLam_\cS$. It is straightforward to check that 
\beq
\label{eq:eigspacecontain}
\spn (\{ \fZ_{\bLam_j}\}_{j=1}^J) \subset \fZ_{\bLam_\cS}.
\eeq

To prove the reverse containment, use the assumption that $\cS$ is slice mix invertible and the fact that $\cT$ has rank $R$ to conclude $\cS$ has rank $R$. Letting $\{\bLam_{\cS,1}, \dots, \bLam_{\cS,\ell}\}$ be the set of JGE values of $\cS$, it follows from Theorem \ref{theorem:rankvsmult} together with \cite[Lemma 2.2]{EvertD19} that the collection of spaces $\{\fZ_{\bLam_{\cS,i}}\}_{i=1}^\ell$ spans $\K^R$ and also that these spaces pairwise have trivial intersection. Again using Theorem \ref{theorem:rankvsmult} shows that the JGE vectors of $\cT$ span $\K^R$. Thus, since each JGE vector of $\cT$ lies in some $\fZ_{\bLam_{\cS,i}}$, the containment in equation \eqref{eq:eigspacecontain} must in fact be an equality.  \end{proof}

Proposition \ref{proposition:FormOfEigenspaces} has the implication that one may, for example, compute joint generalized eigenvectors of $\cT$ by computing intersections of generalized eigenspaces of various different subpencils of $\cT$. In this approach one must compute enough generalized eigenspaces so that all individual generalized eigenvectors may be recovered by taking intersections of generalized eigenspaces.

\subsubsection{The generalized eigenvalue decomposition for CPD computation}

We now briefly recall the main idea behind the GEVD \cite{DD14} algorithm for CPD computation. Suppose $\cT = \cpd \in \R^{R \times R \times K}$ has rank $R$ and is slice mix invertible, and let $\bV \in {K \times 2}$ be a generic matrix. Additionally set $\cS = \cT \cdot_3 \bV^\teT = [\![ \bA,\bB,\bV^\teT \bC]\!]$. Since $\bV$ is generic, the tensor $\cS$ has rank $R$ and is slice mix invertible. It follows, e.g. from Theorem \ref{theorem:rankvsmult}, that the factor matrix $\bB$ of $\cT$ can be computed by computing a generalized eigenvalue decomposition of $\cS$. Having solved for $\bB$, one may recover $\bA$ and $\bC$ by solving for $\bC \odot \bA$ in
\[
\bT_{[2;3,1]} = \bB (\bC \odot \bA)^\teT,
\]
then computing rank-$1$ decompositions of the columns of $\bC \odot \bA$. 

The key difficulty in this algorithm that we aim to address is that the step which maps $\cT$ to $\cS = \cT \cdot_3 \bV^\teT$ is not stable \cite{BBV19}. Intuitively, this is related to the fact that the separation between the columns of $\bV^\teT \bC$ is expected to be much smaller than the separation between the columns of $\bC$. Indeed when one projects a collection of vectors in $\K^K$ to $\K^2$, the angle between those vectors cannot increase and is expected to decrease.\footnote{In fact, if one projects a collection of vectors from $\R^K$ to $\R^2$, then the angle between some pair of vectors in the collection must decrease unless the collection is contained in a two-dimensional subspace of $\R^K$. }

The original GEVD algorithm was communicated in 1970 by Harshman who credits the algorithm to Jennrich \cite{Harsh70}. The GEVD algorithm has since been rediscovered may times, e.g., see \cite{SK90,LRA93,FBK94}. We note that aside from GEVD, there are several popular semi-algebraic methods for computing the CPD of a tensor, e.g., see \cite{delathauwer2006link,PTC15,RH13}.

\subsection{Guide to the reader}

Section \ref{sec:GESDmethod} explains the generalized eigenspace decomposition (GESD) algorithm for computation of a canonical polyadic decomposition. In Section \ref{sec:DetermBound}, we give a bound that is predictive for the maximal noise level that does not inhibit accurate decomposition. In particular, we give a deterministic bound using perturbation theory for JGE values. Sections \ref{sec:GESDPerformance} and \ref{sec:CRboundPerformance} present numerical experiments for the GESD algorithm. In Section \ref{sec:GESDPerformance} we compare the performance (in terms of timing and accuracy) of GESD to the classical GEVD algorithm. In Section \ref{sec:CRboundPerformance} we compare the accuracy of the GESD algorithm to the bound computed in Section and \ref{sec:DetermBound}.

\section{The GESD algorithm}
\label{sec:GESDmethod}

We now give a formal description of the GESD algorithm for computation of a canonical polyadic decomposition. The main step of the algorithm is to use generalized eigenspaces to recursively split a given tensor into subtensors of reduced rank until the subtensors all have rank one  (up to noise). The CPD of each  (noisy) rank one tensor may be computed by a  best rank one approximation, and the CPD for the full tensor is obtained by combining the decompositions of these rank one tensors.

 Let $\cT$ be an $R \times R \times K$ slice mix invertible tensor of rank $R$ with CPD $\cT= \cpd$. Let $\bV \in \K^{K \times 2}$ be a column-wise orthonormal matrix\footnote{Column-wise orthonormal matrices are used for numerical reasons.} of size $K \times 2$ and set $\cS =\cT \cdot_3 \bV^{\text{T}}$. The GESD algorithm begins by computing a collection of generalized eigenspaces $\fZ_1, \dots, \fZ_L$ for the generalized eigenvalue problem
\beq
\label{eq:gevd}
\bS_i \bz= \blam_i  \bq \mathrm{\ for \ } i=1,2
\eeq
such that $\spn (\{\fZ_\ell\}_{\ell=1}^L)=\K^R$. In each step the GESD algorithm computes as many well-separated eigenspaces as possible; however, to simplify exposition we will take $L=2$.

Say the generalized eigenvalues  of $\cS$ are split into two distinct groups of size $j_1$ and $j_2$  where $j_1+j_2=R$. The corresponding eigenspaces may then for instance be computed with a $\bQ\bZ$ decomposition\footnote{The $\bQ \bZ$ decomposition is the generalization of the Schur decomposition to matrix pencils, e.g., see \cite[Chapter VI, Theorem 1.9]{SS90}.} of the matrix pencil $(\bS_1,\bS_2)$ \cite{camps2019phd}. In particular, the $\bQ \bZ$ decomposition computes unitary matrices $\bQ,\bZ \in \K^R$ such that
\beq
\label{eq:QZ}
\bS_1 = \bQ \bX_1 \bZ^\teH \mathand \bS_2 = \bQ \bX_2 \bZ^\teH
\eeq
where $\bX_1$ and $\bX_2$ are (quasi\footnote{When working over $\R$, if $\bS$ has complex generalized { eigenvectors}, then since $\bQ$ and $\bZ$ have real entries, there will be $2 \times 2$ block matrices on the diagonal of $\bX_1$ and $\bX_2$. E.g., see \cite[Theorem 1.2]{K05}.}) upper triangular matrices with the classical generalized eigenvalues of the pencil on the diagonal. 

If the decomposition is ordered so that the $j_1$ generalized eigenvalues in the first cluster appear first on the diagonal of $\bX_1$ and $\bX_2$, then the first $j_1$ columns of $\bZ$ give us a matrix $\bZ_1$ whose columns form a basis for the corresponding generalized eigenspace of $\cS$. Similarly, we obtain a matrix $\bZ_2$ whose columns form a basis for the eigenspace corresponding to the next $j_2$  generalized eigenvalues by reordering the terms in the $\bQ\bZ$ decomposition so that these generalized eigenvalues appear first on the diagonal of $\bX_1$ and $\bX_2$. 

Using Proposition \ref{proposition:FormOfEigenspaces} together with Theorem \ref{theorem:rankvsmult} shows that the matrices $\bZ_1$ and $\bZ_2$ that we have obtained are of the form $
\mathbf{Z}_i=\bB^{-\teT} \bE_i$ where the $\bE_i$ have the property that the $j$th row of $\bE_1$ is nonzero if and only if the $j$th row of $\bE_2$ is a zero row, and vice versa.  Define the tensors $\cT^1$ and $\cT^2$ by 
\[
\cT^i= \cT \cdot_2 \bZ_i^\teT  =[\![\bA,\bZ_i^\teT \bB,\bC]\!]=[\![\bA,\bE_i^\teT,\bC]\!]
\]
and for $i=1,2$ set $R_i=\rank (\cT^i)$. Since the $\bE_i^\teT$ have zero columns corresponding to the zero rows of $\bE_i$, we have that $R_1+R_2=\rank (\cT)$. Moreover, suppose $\cT^i$ has CPD $[\![\bA_i, \check{\bE}_i,\bC_i]\!]$ for $i=1,2$. Here the matrix $\check{\bE}_i \in \K^{R \times R_i}$ is the matrix obtained by removing the zero columns of $\bE_i^\teT$. Then up to permutation and scaling, one has $\bA= ( \bA_1 \ \ \bA_2 )$.  Similarly, the factor matrix $\bC$ of $\cT$ has the form $\bC=(\bC_1 \ \bC_2)$ up to scaling and the same permutation as $\bA$. 

This procedure partitions the rank one terms of $\cT$ into two distinct collections so that one collection gives a decomposition for the tensor $\cT_1$ while the other collection gives a decomposition for $\cT_2$. The number of rank one terms in each collection, i.e., the rank of $\cT_1$ and of $\cT_2$, is determined by the number of generalized eigenvalues in each cluster.

The above methodology may be recursively applied to each $\cT^i$ to further reduce their ranks until one obtains a collection of tensors all having rank one (up to noise). The columns of the factor matrices $\bA$ and $\bC$ may then  be approximated by computing rank one approximations for each of these  (approximately) rank one tensors, see Section \ref{sec:rankoneterms} below for details.

Once we have computed the factor matrices $\bA$ and $\bC$, we may compute $\bB$ by solving
\[
\bT_{[2;3,1]}=\bB(\bC \odot \bA)^\teT.
\]

We emphasize that the GESD algorithm is guaranteed to recover a CPD of $\cT$ in the noiseless setting under the assumption that $\cT$ is an $R \times R \times K$ slice mix invertible tensor of rank $R$. Note that the CPD of such a tensor is not necessarily unique. In particular, if $\cT= \cpd$ meets this assumption but the factor matrix $\bC$ has a repeated column, i.e. if $\cT$ has a repeated joint generalized eigenvalue, then GESD can successfully recover a CPD of $\cT$ even though $\cT$ does not have a unique CPD.

\subsection{Rank one terms}
\label{sec:rankoneterms}
Consider the case where $\bZ_1 \in \K^{R \times 1}$ where $\bZ_1$ spans an eigenspace of dimension one. In this case $\bZ_1$ must have the form
\[
\bZ_1=\beta\bB^{-\teT}\vec{e}_\ell^{\teT}
\]
where $\vec{e}_\ell$ is a standard basis vector and $\beta \in \K$ is some constant. Thus we have
\[
\cT  \tmprod{2} \bZ_1^\teT = \cp{\bA, \beta \vec{e}_\ell,\bC}= \cp{\vec{a}_\ell, \beta, \vec{c}_\ell}.
\]
In this case we may compute $\vec{a}_\ell$ and $\vec{c}_\ell$ by computing a rank one approximation of the matrix $\cT \tmprod{2} \bZ_1^\teT$. In practice one may verify that such a tensor has rank one (up to noise) by computing the multilinear singular values of $\cT \cdot_2 \bZ_1^T$.

\subsection{Choice of slices}

We observe that the average performance of the GESD algorithm is improved by first considering a pencil formed by taking the first two frontal slices of the core of the MLSVD of the current subtensor on each splitting step. Intuitively, this is related to the fact that this matrix pencil {explains more of the  overall sum of squared entries than any other matrix pencil which can be formed from the frontal slices of the tensor. 

We emphasize that this strategy is not always successful. Indeed, there are many settings where the generalized eigenvalues obtained by considering the first two slices of the core of an MLSVD are poorly separated. When this occurs it is necessary to consider other matrix pencils. See Section \ref{sec:BadMLSVDSlices} for an example. In the case that a successful split is not obtained by considering slices of the MLSVD, we move on to considering random linear combinations of frontal slices.

Note that an MLSVD is always computed for each subtensor to address other potential issues (e.g., repeated factor matrix columns). Furthermore computing MLSVDs of the subtensors is cheap relative to other steps of the algorithm. Thus first considering slices of the MLSVDs has little impact on overall computational cost and on average significantly improves accuracy. 

Pseudo code for the GESD algorithm is found in Algorithm \ref{alg:mainpsuedo} and Algorithm \ref{alg:eigspacepsuedo}.

\begin{algorithm}
	\label{alg:mainpsuedo}
	\SetKwInOut{Input}{Input}
	\SetKwInOut{Output}{Output}
	\underline{function $\texttt{cpd\_gesd}$} $(\mathcal{M},R,\texttt{Random})$\\
	\Input{Tensor $\mathcal{M}$, rank $R$, boolean $\texttt{Random}$}
	\Output{$\cp{\mathbf{A},\mathbf{B},\mathbf{C}}$}
	$K=\min (\text{size}(\mathcal{T},3),R)$\\
	$[\mathbf{V}_1,\mathbf{V}_2,\mathbf{V}_3,\mathcal{Y}] = \texttt{mlsvd}(\mathcal{M},[R,R,K])$;
	{\begin{flushright} \% Compute MLSVD of $\cT$/ Compress \end{flushright}}
	\eIf{$R=1$}
	{			
		$\mathbf{A}_\text{core} = y_{111}$;\  $\mathbf{B}_\text{core} = 1$;\   $\mathbf{C}_\text{core} = 1$; \hfill { \% Trivial case}
	}
	{
		$[\mathbf{Z}_1,\ldots, \mathbf{Z}_N] = \texttt{qz\_split}(\mathcal{Y},\texttt{Random})$; {  \begin{flushright}\% Split JGE spaces into $N$ separated subgroups (Algorithm \ref{alg:eigspacepsuedo})\\ \end{flushright}}
		\For{$i = 1:N$}
		{
			$\mathcal{Y}^i = \mathcal{Y} \tmprod{2} \mathbf{Z}_i^\text{T}$; { \hfill\% Reduce tensor to $ N$th subgroup}\\
			$R_i=\texttt{size}(\mathbf{Z}_i,2)$\hfill \% Specify rank for $N$th subgroup\\
			$\cp{\mathbf{A}_i,\mathbf{B}_i,\mathbf{C}_i} = \texttt{cpd\_gesd}(\mathcal{Y}^i,R_i);$ {\begin{flushright} \% Recursion on smaller tensor and smaller rank \end{flushright}}		
		}
		$\mathbf{A}_\text{core} = [\mathbf{A}_1,\ldots,\mathbf{A}_N]$;\
		$\mathbf{C}_\text{core} = [\mathbf{C}_1,\ldots,\mathbf{C}_N]$; {  \begin{flushright} \% Build $\mathbf{A}$ and $\mathbf{C}$ from subblocks\\ \end{flushright}}
		$\mathbf{B}_\text{core} = \mathbf{Y}_{[2;3,1]} /(\mathbf{C}\kr \mathbf{A})^\text{T}$;{  \hfill \% Find $\mathbf{B}$ by solving linear system}\\
		
	}
	$\mathbf{A} = \mathbf{V}_1\mathbf{A}_\text{core}$;\ 
	$\mathbf{B} = \mathbf{V}_2\mathbf{B}_\text{core}$;\ 
	$\mathbf{C} = \mathbf{V}_3\mathbf{C}_\text{core}$; {\begin{flushright} \% Undo the MLSVD / Expand \end{flushright}}
	\caption{GESD recursive step}
\end{algorithm}
{}

\begin{algorithm}
	\label{alg:eigspacepsuedo}
	\SetKwInOut{Input}{Input}
	\SetKwInOut{Output}{Output}
	
	\underline{function $\texttt{qz\_split}$} $(\mathcal{Y},\texttt{Random})$\\
	\Input{ MLSVD core tensor $\mathcal{Y}$, boolean \texttt{Random}}
	\Output{$ N$ matrices $\mathbf{Z}_n$ spanning JGE spaces}
	$R = \texttt{size}(\mathcal{Y},1)$\\
	$N = 0$\\
	\While{$N<2$}{
	\eIf{\texttt{Random}}{
		$\mathcal{S} = \mathcal{Y} \tmprod{3} \texttt{qr}(\texttt{randn}(R,2))^\teT$; \hfill { \% Construct random pencil} \\
	}
	{
		\text{Set indices $k_1$ and $k_2$ of next set of MLSVD slices}\\
		$\mathcal{S} = \mathcal{Y}_{:,:,[{k_1},{k_2}]}$; \hfill { \% Take MLSVD slices as pencil} \\
	}
	$[\mathbf{X}_1,\mathbf{X}_2,\mathbf{Q},\mathbf{Z}] = \texttt{qz}(\mathbf{S}_{::1},\mathbf{S}_{::2})$; {\begin{flushright} { \% Compute QZ decomposition of pencil}\end{flushright}}
	\text{Order eigenvalues $x_{2,ii}/ x_{1,ii}$}\\
	\text{Compute chordal gaps between adjacent pairs $(x_{1,ii},x_{2,ii})$}\\
	\text{Set $N$ to be the number of chordal gaps above threshold}\\
}
	\text{Split eigenvalues into  $N$ clusters according to gaps above threshold}\\
	\For{$n = 1: N$}
	{
		\text{Define binary vector $\vec{i}_n$ indicating the $n$th cluster.}\\
		$[\mathbf{X}_{1,n},\mathbf{X}_{ 2,n},\mathbf{Q}_{n},\mathbf{Z}_{n}] = \texttt{ordqz}(\mathbf{X}_1,\mathbf{X}_2,\mathbf{Q},\mathbf{Z},\vec{i}_{n})$;  { \begin{flushright} \% Modify QZ decomposition such that $n$th cluster is at top\\ \end{flushright}}
		$\mathbf{Z}_{n} = \left({\mathbf{Z}_{n}}\right)_{:,1:\texttt{sum}(\vec{i}_{n})}$ \hfill { \% Extract JGE space of $n$th cluster} 
	}
	
	\caption{Splitting eigenspaces (over $\R$)\protect\footnotemark}
	
\end{algorithm}

\subsection{Handling of complex eigenvectors when working over $\R$}
\label{sec:ComplexEigs}

Our explanation of the GESD algorithm above focuses on the noiseless setting. An important aspect of this setting is that the joint generalized eigenvectors of a slice mix invertible real rank $R$ tensor $\cT \in \R^{R \times R \times K}$ are necessarily real valued \cite{EvertD19}, and, as a consequence, the subpencils of such a tensor have real valued generalized eigenvectors. In the noisy setting, however, it is possible that a subpencil of a perturbation $\cY=\cT+\cN$ has complex valued generalized eigenvectors.

For example, suppose $\cT \in \R^{R \times R \times K}$ is a slice mix invertible tensor of real rank $R$, and suppose $\fZ$ is a JGE space of dimension $2$. Let $\bZ \in \K^{R \times 2}$ be a  matrix whose columns form a basis for $\fZ$ and let $\cS$ be an MLSVD of the tensor $\cY  \tmprod{2} \bZ^\teT$. Then $\cS$ is a $2 \times 2 \times 2$ tensor. In the noiseless setting (i.e., if $\cY=\cT$), if $\cS$ is slice mix invertible, then $\cS$ necessarily has a real rank $2$ and both generalized eigenvectors of $\cS$ are real valued. However, with sufficient noise, it is possible that $\cS$ has real rank $3$. This corresponds to the case where the generalized eigenvectors of $\cS$ are complex valued.

If $\cS$ has rank $3$, then $\cS$ does not have a best real rank $2$ approximation \cite{SL08}, so $\cS$ cannot be further decomposed and the true real rank $R$ decomposition of the original tensor cannot be computed by the GESD algorithm at this point. However, the factor matrix $\bA$ (resp. $\bC$) in the CPD for the original tensor will have two columns which (approximately) lie in the mode-$1$ (resp. mode-$3$) subspace of the tensor $\cY \cdot_2 \bZ^{\text{T}}$.

\footnotetext{When working over $\C$ there is no natural way to sort the generalized eigenvalues. As a consequence one must compute the chordal distance between all pairs of generalized eigenvalues in order to cluster the generalized eigenvalues into well-separated collections.}

At this point one may accept that these particular columns of $\bA$ and $\bC$ cannot be computed and replace them with bases for the two-dimensional subspaces to which they belong. Alternatively, one could use the results from the GESD algorithm (with the bases for these two-dimensional subspaces used as those two columns for $\bA$ and $\bC$) as an initialization { in an optimization routine for computing the full CPD}. While it is unlikely that these basis vectors are the true columns of $\bA$ or $\bC$, the optimization routine will at least be initialized with these columns in the correct two-dimensional subspace which is a significant improvement over a random initialization.

\section{ A deterministic bound for distinct eigenvalue clusters}
\label{sec:DetermBound}
Our main goal in this section is to provide a deterministic bound for the existence of a subpencil of the tensor of interest which has two distinct clusters of generalized eigenvalues which cannot have coalesced, a notion we now explain.

Suppose $\cT \in \K^{R \times R \times 2}$ is a matrix pencil with distinct generalized eigenvalues. Then for sufficiently small perturbations $\cN$, the generalized eigenvalues and generalized eigenvectors of $\cT+\cN$ are continuous in $\cN$. However, if the perturbation $\cN$ is large enough so that the matrix pencil $\cT+\cN$ has a repeated generalized eigenvalue, i.e., so that at least two generalized eigenvalues of $\cT+\cN$ have coalesced, then the path of corresponding generalized eigenvectors  cannot be unambiguously continued, as after the generalized eigenvalues separate there is ambiguity in which generalized eigenvalue corresponds to which generalized eigenvector. 

Although it is no longer possible to  unambiguously continue the paths of generalized eigenvectors themselves after generalized eigenvalues have coalesced, the path of the corresponding generalized eigenspace is still continuous  and unambiguous in $\cN$ so long as the underlying generalized eigenvalues do not coalesce with other generalized eigenvalues. Thus guaranteeing the existence of two distinct clusters of generalized eigenvalues which cannot have coalesced guarantees that $\cT+\cN$ has at least two distinct generalized eigenspaces which are still continuous in the perturbation $\cN$. These two generalized eigenspaces are still closely related to the original generalized eigenspaces of $\cT$. This idea is formalized for matrix pencils in \cite{DK87}. Also see \cite{Kato95} for the matrix case.

We note that the bound presented in this section can not only guarantee the existence of clusters of generalized eigenvalues which have not coalesced, but can also guarantee a minimum separation between the clusters of generalized eigenvalues. As previously mentioned, the separation between clusters of generalized eigenvalues heavily impacts the  accuracy of computing corresponding generalized eigenspaces, so in this way our bounds have significant implications for the accuracy of a given step of the GESD algorithm.

When discussing  our deterministic bound we restrict to the real setting. That is, in this section $\cT \in \R^{R \times R \times K}$ is a slice mix invertible real tensor which has real rank $R$.

We now introduce some notation. Given a matrix $\bM$, let $\sigma_{\min}(\bM)$ denote the smallest singular value of $\bM$. Define the \df{Grassmannian} $\Gr(1,\K^2)$ to  be the collection of subspaces of $\K^2$ of dimension $1$. Note that any element of $\Gr(1,\R^2)$ may be uniquely represented in the form $\spn (\cos(\theta),\sin(\theta))$ for some $\theta \in [0,\pi)$. Thus given a subspace $\bLam \in \Gr(1,\R^2)$ we define
\[
{\theta(\bLam) \in [0,\pi) \qquad  \mathrm{by } \qquad \bLam= \spn\Big(\big(\cos(\theta(\bLam)),\sin(\theta(\bLam))\big)\Big). }
\]
Additionally, for a nonzero vector $\blam \in \R^2$, define 
\[
\theta (\blam)= \theta (\spn(\blam)).
\]

The metric we use for the Grassmannian is the chordal metric. Given two subspaces $\bLam_1,\bLam_2 \in \Gr(1,\K^2)$ define the \df{chordal metric} between $\bLam_1$ and $\bLam_2$, denoted $\chi(\bLam_1,\bLam_2)$ to be the sine of the (principal) angle between $\bLam_1$ and $\bLam_2$. It is straightforward to check that $\Gr(1,\K^2)$ is a path connected metric space under the chordal metric. We extend the definition of $\chi$ to $\K^2\backslash \{0\}$ by identifying nonzero vectors in $\K^2$ with their spans. That is, given $\blam_1,\blam_2 \in \K^2 \backslash \{0\}$ we have
\[
\chi (\blam_1,\blam_2):= \chi(\spn(\blam_1),\spn(\blam_2)).
\]
See \cite{YL16} for further discussion of the chordal metric and distance in Grassmanians.

Lastly, for a tensor $\cT \in K^{R \times R \times K}$, define the \df{spectral norm} of $\cT$, denoted $\spnorm{\cT}$  by
\[
\|\cT\|_{\mathrm{sp}}:= \max_{\|\bx\|_2=\|\by\|_2=\|\bz\|_2=1} |\cT \cdot_1 \bx^\teT \cdot_2 \by^\teT \cdot_3 \bz^\teT|
\]
where the vectors $\bx,\by,\bz$ are of compatible dimension. The spectral norm of $\cT$ is equal to the Frobenius norm of a best rank $1$ approximation of $\cT.$ See \cite{FL18,FMPS13,FO14,QHX21} for further discussion.

Our first step towards a bound for an $R \times R \times K$ tensor to have a subpencil in which generalized eigenvalues have not coalesced is to give a bound which guarantees that a matrix pencil (i.e., an $R \times R \times 2$ tensor) has generalized eigenvalues which have not coalesced.

\begin{theorem}
\label{theorem:RegPertJEigBound}
Let $\cT$ and $\cY$ be tensors of size $R \times R \times 2$. Assume that $\cT$  is slice mix invertible and has real rank $R$ with CPD $\cT = \cpd$, where the columns of $\bC$ are normalized to have unit norm.

Let $\{\spn(\blam_1), \dots, \spn(\blam_R)\}$ be the  generalized eigenvalues of $\cT$ where the indices are ordered so that 
\[
0 \leq \theta (\blam_r) \leq \theta (\blam_{r+1}) < \pi \quad \mathrm{for } \quad r=1,\dots, R-1,
\]
and assume at least $J$ of the $\blam_r$ are distinct. Define $\blam_{R+1}=\blam_1$ and let $\delta$ be the $J$th largest element of the set
\[
\left\{ \frac{\chi(\blam_r,\blam_{r+1})}{2}
\right\}_{r=1}^R.
\]
Set $\epsilon_1= \sigma_{\min} (\bA) \sigma_{\min}(\bB) \delta$  and let $\epsilon_2$ be the distance in the spectral norm from $\cT$ to the nearest pencil which is not slice mix invertible. If
\beq
\label{eq:DistanceReq}
\spnorm{\cT-\cY} < \min\{\epsilon_1,\epsilon_2\},
\eeq
then the tensor $\cY$ has at least $J$ distinct { (possibly complex valued)} generalized eigenvalues. 

Furthermore, if equation \eqref{eq:DistanceReq} is satisfied and one sets $\spnorm{\cT-\cY}  = \epsilon$, then the generalized eigenvalues of $\cY$ may be divided into $J$ nonempty sets $
\{\mathscr{L}^j\}_{j=1}^J$ 
such that
\beq
\label{eq:GenEigSetGaps}
\min_{j \neq \ell} \chi (\mathscr{L}^j,\mathscr{L}^\ell) \geq 2\delta-\frac{2\epsilon}{\sigma_{\min}(\bA) \sigma_{\min}(\bB)}  > 0. 
\eeq

\end{theorem}
\begin{remark}
See the upcoming Lemma \ref{lemma:Eps2lower} for discussion of how to compute a lower bound for the quantity $\epsilon_2$ appearing in Theorem \ref{theorem:RegPertJEigBound}. In the case $J=R$ one always has $\epsilon_1 \leq \epsilon_2$, see \cite[Proposition 4.4]{EvertD19}.
\end{remark}

The proof of Theorem \ref{theorem:RegPertJEigBound} follows a simple idea. Using perturbation theoretic results for a matrix pencil, one may show that there are at least $J$ distinct connected regions in $\Gr(1,\C^2)$ such that each region contains a generalized eigenvalue of $\cY$ and such that all generalized eigenvalues of $\cY$ are contained in the union of these disjoint regions. Roughly speaking, the ``real parts" of these regions lie between each of the $J$ largest chordal gaps between the ordered generalized eigenvalues of $\cT$. To formalize this idea we introduce the notion of a Bauer--Fike ball and a Bauer--Fike region around a given generalized eigenvalue of $\cT$.

 Let the tensor $\cT=\cpd$ and the generalized eigenvalues $\{\spn(\blam_r)\}_{r=1}^R$ be as in the statement of Theorem \ref{theorem:RegPertJEigBound}. In particular, $\cT$ is slice mix invertible with $\R$-rank $R$ and the columns of $\bC$ have unit norm. We define the (real or complex) \df{Bauer--Fike ball}, for the parameter $\epsilon$ around a given generalized eigenvalue $\spn(\blam_r)$ of $\cT$, denoted $\bfballK (\blam_r) \subset \Gr(1,\K^2)$, by
\[
\bfballK (\blam_r)=\left\{\spn(\bgam)\   \in \  \Gr(1,\K^2) : \bgam \in \K^2/\{0\} \mathrm{\ and\ }  \chi(\bgam,\blam_r) \leq \frac{\epsilon}{\sigma_{\min}(\bA) \sigma_{\min} (\bB)} \right\}.
\]

Having defined Bauer--Fike balls\footnote{For clarity we emphasize that a Bauer--Fike ball for the parameter $\epsilon$ around a generalized eigenvalue $\bLam$ of $\cT$ is simply the closed neighborhood in $\Gr(1,\K^2)$ of radius $\frac{\epsilon}{\sigma_{\min}(\bA) \sigma_{\min} (\bB)}$ centered at $\bLam$ . We introduce this notation to streamline the proof of the upcoming Proposition \ref{proposition:BFrestrictR}, and to clearly illustrate the connection between these sets.}, we define the Bauer--Fike region $\bfregK(\blam_r)$ to be the connected component of the union of Bauer--Fike balls for $\cT$ which contains $\spn(\blam_r)$. That is, we set 
\[
\mathscr{D}^{\K} = \cup_{r=1}^R \bfballK (\blam_r) \subset \Gr(1,\K^2),
\]
and let $\mathscr{D}^{\K} = \cup_{r=1}^{R'} \mathscr{C}_r^\K$ be a partitioning of $\mathscr{D}^\K$ into $R' \leq R$ disjoint connected components (in the topology induced by the chordal metric). The Bauer--Fike region $\bfregK(\blam_r)$ is the unique connected component of $\mathscr{D}^\K$ which contains $\spn(\blam_r)$.

For use in the proof of the upcoming Proposition \ref{proposition:BFrestrictR}, we also provide the following equivalent definition of a Bauer--Fike region. For each $N = 0,1,2,\dots$ define
\[
\bfballNplK (\blam_r)= \cup_{\bfballNK(\blam_r) \cap \bfballK(\blam_\ell) \neq \emptyset} \bfballK(\blam_\ell).
\]
That is, $\bfballNplK(\blam_r)$ is the union over $\ell$ of Bauer--Fike balls $\bfballK(\blam_\ell)$ such that $\bfballK(\blam_\ell)$ has nontrivial intersection with $\bfballNK(\blam_r)$. The Bauer--Fike region $\bfregK(\blam_r)$ is defined by
\[
\bfregK(\blam_r)=\cup_{N=1}^\infty \bfballNK(\blam_r).
\]
It is not difficult to show that there is an integer $N$ such that 
\[
\bfregK(\blam_r)= \bfballNK(\blam_r).
\]
In particular one can take $N=R$.

The complex Bauer--Fike regions for a given tensor $\cT$ are useful in that they give us control over the locations of the generalized eigenvalues of a perturbation of $\cT$. Assume $\epsilon$ is small enough so that $ \spnorm{\cT-\cY} \leq \epsilon $ implies that $\cY$ is slice mix invertible. Then following the argument in the proof of \cite[Theorem 3.5]{EvertD19} shows that the number of generalized eigenvalues of $\cY$ contained by each complex Bauer--Fike region for the parameter $\epsilon$ is equal to the number of overlapping complex Bauer--Fike balls which make up that complex Bauer--Fike region. That is, for a fixed index $r$, if there are exactly $m$ indices $\ell$ (counting the case $r=\ell$) such that
\[
\bfregC(\blam_r) \cap \bfballC(\blam_\ell) \neq \emptyset,
\]
and if $\spnorm{\cT-\cY} \leq \epsilon $, then the complex Bauer--Fike region $\bfregC(\blam_r) $ contains exactly $m$ of the generalized eigenvalues of the tensor $\cY$ counting algebraic multiplicity.

Our goal in the proof of Theorem \ref{theorem:RegPertJEigBound} will be to show that $\cT$ has at least $J$ distinct complex Bauer--Fike regions for the parameter $\epsilon$. {Following the above discussion shows that if} $\cT$ and $\cY$ satisfy the assumptions of Theorem \ref{theorem:RegPertJEigBound}, then $\cY$ has at least $J$ distinct clusters of generalized eigenvalues which cannot have coalesced in a perturbation from $\cT$ to $\cY$. To show that $\cT$ has $J$ distinct Bauer--Fike regions, we will make use of the fact that Bauer--Fike regions are connected subsets of $\Gr(1,\K^2)$.

\begin{lemma}
\label{lemma:BFballsConnected}
Let $\cT \in \R^{R \times R \times 2}$ be a slice mix invertible $\R$-rank $R$ tensor with generalized eigenvalues $\frak{L}=\{\blam_r\}_{r=1}^R$. Then for each $r$, the Bauer--Fike region $\bfregK(\blam_r) \subset \Gr (1,\K^2)$ is path connected in the topology on $\Gr(1,\K^2)$ given by the chordal metric.
\end{lemma}

Another technical part of our proof will be to restrict from considering complex Bauer--Fike regions to considering real Bauer--Fike regions. We emphasize that the generalized eigenvalues of $\cY$ can only be guaranteed to lie in the union of the complex Bauer--Fike regions of $\cT$, thus it is necessary to use complex Bauer--Fike regions. However, complex Bauer--Fike regions are more difficult to work with due to a lack of any notion of ordering in $\Gr(1,\C^2)$,  hence our desire to reduce to working over $\R$.  The following proposition makes this reduction possible.

\begin{proposition}
\label{proposition:BFrestrictR}
Let $\cT \in \R^{R \times R \times 2}$ be a slice mix invertible tensor having real rank $R$. Let $\{\spn (\blam_r)\}_{r=1}^R$ be the set of generalized eigenvalues of $\cT$ with $\blam_r \in \R^2$ for each $r  = 1,\dots, R$. Then for any $r_1, r_2$ one has
\beq
\label{eq:ComplexIntIFFRealInt}
\bfballC (\blam_{r_1}) \cap \bfballC (\blam_{r_2}) = \emptyset \quad \mathrm{if \ and \ only \ if }  \quad \bfballR (\blam_{r_1}) \cap \bfballR (\blam_{r_2}) = \emptyset.
\eeq
Additionally,
\[
\bfregC (\blam_{r_1}) = \bfregC (\blam_{r_2}) \quad \mathrm{if \ and \ only \ if }  \quad \bfregR (\blam_{r_1}) = \bfregR (\blam_{r_2})
\]
It follows that the number of distinct complex Bauer--Fike regions for $\cT$ is exactly equal to the number of distinct real Bauer--Fike regions for $\cT$. 
\end{proposition}

\begin{proof}
Let $\cT$ and $\blam_{r_1}$ and $\blam_{r_2}$ be as in the statement of the lemma. It is clear that $\bfballC (\blam_{r_1}) \cap \bfballC (\blam_{r_2}) = \emptyset$ implies $\bfballR (\blam_{r_1}) \cap \bfballR (\blam_{r_2}) = \emptyset$. On the other hand, using the fact that a vector $\bzeta$ which minimizes the chordal distance to a set of real vectors $\{\blam_1,\blam_2\}$ may be obtained by taking $\bzeta$ to be a real vector that bisects the angle between $\blam_1$ and $\blam_2$, it is straightforward to show that if $\bfballC (\blam_{r_1}) \cap \bfballC (\blam_{r_2}) \neq \emptyset$, then $\bfballR (\blam_{r_1}) \cap \bfballR (\blam_{r_2}) \neq \emptyset$. In particular, if $\bfballC (\blam_{r_1}) \cap \bfballC (\blam_{r_2}) \neq \emptyset$, then 
\[
\spn(\bzeta) \in \bfballC (\blam_{r_1}) \cap \bfballC (\blam_{r_2}) \mathrm{\ hence\ } \spn(\bzeta) \in \bfballR (\blam_{r_1}) \cap \bfballR (\blam_{r_2}).
\]
Thus we have $\bfballC (\blam_{r_1}) \cap \bfballC (\blam_{r_2}) = \emptyset$ if and only if $\bfballR (\blam_{r_1}) \cap \bfballR (\blam_{r_2}) = \emptyset$.

Using equation \eqref{eq:ComplexIntIFFRealInt} together with the above, one may show by induction that for all $N \in \N$ and all $r_1,r_2 \in \{1,\dots,R\}$ one has
\beq
\label{eq:CcontainIFFRcontain}
\bfballC(\blam_{r_2}) \subset \bfballNC(\blam_{r_1}) \quad \mathrm{if \ and \ only \ if }  \quad \bfballR(\blam_{r_2}) \subset \bfballNR(\blam_{r_1}).
\eeq
Furthermore, it is straightforward to check that 
\[
\bfregK(\blam_{r_1}) = \bfregK(\blam_{r_2})
\]
if and only if there is an integer $N$ such that
\[
\bfballK(\blam_{r_2}) \subset \bfballNK(\blam_{r_1}).
\]
Combining this fact with equation \eqref{eq:CcontainIFFRcontain} completes the proof.
\end{proof}

\subsection{Proof of Theorem \ref{theorem:RegPertJEigBound}}

We have now laid the necessary groundwork to give the proof of Theorem \ref{theorem:RegPertJEigBound}. The proof is technical, but, as previously mentioned, the main idea behind the proof is simple. Before giving technical details, we give an overview of the role played by Bauer--Fike regions. 

As previously discussed, each distinct Bauer--Fike region for $\cT$ contains a generalized eigenvalue of $\cY$ and all the generalized eigenvalues of $\cY$ are contained in some Bauer--Fike region of $\cT$. This fact is proved in the tensor spectral norm by \cite[Theorem 4.3]{EvertD19}; however, in a different norm (which upper bounds the tensor spectral norm) this fact is classical, e.g. see \cite[Chapter VI, Theorem 2.7]{SS90}. By using this fact together with Proposition \ref{proposition:BFrestrictR} to show that $\cY$ has at least $J$ distinct generalized eigenvalues, it is sufficient to show that $\cT$ has $J$ distinct real Bauer--Fike regions. 

To show that $\cT$ has at least $J$ distinct real Bauer--Fike regions, we divide $\Gr(1,\R^2)$ into $J$ distinct regions such that each region fully contains a real Bauer--Fike region of $\cT$. The desired division of $\Gr(1,\R^2)$ into distinct regions is accomplished by using subspaces of $\R^2$ to cut in half each of the $J$ largest gaps between the ordered generalized eigenvalues of $\cT$.

 Our choice of parameter for our Bauer--Fike balls will guarantee that these cuts cannot be contained in any Bauer--Fike ball of $\cT$, hence, these cuts cannot be contained in any Bauer--Fike region. The fact that a real Bauer--Fike region lying between two adjacent cuts is distinct from a real Bauer--Fike region lying between another pair of adjacent cuts follows from the path connectedness of the real Bauer--Fike regions. Details follow below.

\begin{proof}
Let $\cT=\cpd$ and $J$ and $\epsilon$ be as in the statement of Theorem \ref{theorem:RegPertJEigBound} and assume $J \geq 2$. We first show that $\cT$ has $J$ distinct complex Bauer--Fike regions. Using Proposition \ref{proposition:BFrestrictR}, it is sufficient to show that $\cT$ has $J$ distinct real Bauer--Fike regions.  

Now, from the ordering of our eigenvalues, there must exist $J$ distinct indices $1 \leq r_1 < \cdots < r_J \leq R$ such that for each $r_j \neq R$ there is a unit vector $\boldeta_j \in \R^2$ satisfying
\[
\theta(\blam_{r_j}) < \theta (\boldeta_j) < \theta(\blam_{r_j+1})
\]
and
\[
 \chi(\boldeta_j,\blam_{r_j})\geq \delta \mathand \chi(\boldeta_j,\blam_{{r_j}+1})  \geq \delta.
\]
Using $\epsilon_1=\sigma_{\min}(\bA) \sigma_{\min}(\bB) \delta$ and $\epsilon < \epsilon_1$ gives
\[
\chi(\boldeta_\ell,\blam_{r_\ell})> \frac{\epsilon}{\sigma_{\min}(\bA) \sigma_{\min}(\bB) } \mathand \chi(\boldeta_\ell,\blam_{{r_\ell}+1})  >  \frac{\epsilon}{\sigma_{\min}(\bA) \sigma_{\min}(\bB) }  .
\]
If the case $r_J = R$ occurs, then we may instead choose a unit vector $\boldeta_J$ such that
\[
 \chi(\boldeta_J,\blam_{R}) >\frac{ \epsilon}{\sigma_{\min}(\bA) \sigma_{\min}(\bB) } \mathand \chi(\boldeta_J,\blam_{1})  > \frac{\epsilon}{\sigma_{\min}(\bA) \sigma_{\min}(\bB) }
\]
and such that either 
\[
\theta(\blam_R) < \theta(\boldeta_J) \quad \quad \mathrm{{ or}} \quad \quad  \theta(\boldeta_J) < \theta(\blam_1).
\]
To simplify exposition, we assume that $\theta(\blam_R) < \theta(\boldeta_J) $.

Note that $\cup_{j=1}^J \spn(\boldeta_j)$ divides $\Gr(1,\R^2)$ into $J$ distinct open connected components. In particular, there exist $J$ disjoint open subsets $\mathscr{G}_1,\dots,\mathscr{G}_J$ of $\Gr(1,\R^2)$ such that
\[
\cup_{j=1}^J \mathscr{G}_j = \Gr(1,\R^2) \backslash \left(\cup_{j=1}^J \spn (\boldeta_j)\right).
\] 

Using the ordering of our generalized eigenvalues together with the definition of the $\boldeta_j$, a straightforward check shows that $\spn(\boldeta_j)$ cannot be contained in a real Bauer--Fike region $\bfregR (\blam_r)$  for any $r=1,\dots,R$. In particular we have
\beq
\label{eq:BFregsInConnectedComponents}
\cup_{r=1}^R \bfregR(\blam_r) \subset \left(\cup_{j=1}^J \mathscr{G}_j \right).
\eeq
Recalling that each Bauer--Fike region is connected and that $\mathscr{G}_j$ are open and disjoint sets, equation \eqref{eq:BFregsInConnectedComponents} then implies that if there are indices $r$ and $j$ such that
\beq
\label{eq:IntersectionImpliesContainment}
\bfregR(\blam_r) \cap \mathscr{G}_j \neq \emptyset, \qquad \text{then} \qquad  \bfregR(\blam_r) \subseteq \mathscr{G}_j.
\eeq

By construction we have
\[
\mathscr{G}_j \cap   \left(\cup_{r=1}^R \bfregR(\blam_r)\right)  \neq \emptyset \qquad \qquad \mathrm{for \ each \ } j=1,\dots,J.
\]
Combining this with equation \eqref{eq:IntersectionImpliesContainment} shows that each $\mathscr{G}_j$ contains some real Bauer--Fike region of $\cT$. In particular, $\cT$ has at least $J$ distinct real Bauer--Fike regions.  An application of Proposition \ref{proposition:BFrestrictR} then shows that $\cT$ has at least $J$ distinct complex Bauer--Fike regions. 

Now let $\cY \in \R^{R \times R \times 2}$ be a tensor which satisfies
\[
\spnorm{\cT-\cY}  \leq \epsilon < \min\{\epsilon_1,\epsilon_2\}.
\]
 Following the argument used to prove \cite[Theorem 2.8]{EvertD19} shows that for each $j=1,\dots,J$ the region  $\bfregC (\blam_{r_j})$ must contain at least one generalized eigenvalue of  $\cY$ and that all the generalized eigenvalues of $\cY$ are contained in some Bauer--Fike region. Therefore, since $\cT$ has at least $J$ distinct Bauer--Fike regions, the tensor $\cY$ has at least $J$ distinct generalized eigenvalues.

It remains to construct the nonempty sets $\{\mathscr{L}^j\}_{j=1}^J$ of generalized eigenvalues of $\cY$ which satisfy equation \eqref{eq:GenEigSetGaps}. To this end, let $\{\spn(\tblam_r)\}_{r=1}^R$ be the set of generalized eigenvalues of $\cY$ and for each $j=1,\dots,J$ define $\mathscr{L}^j$ by
\[
\mathscr{L}^j = \left\{ \spn (\tblam_\ell) \  : \  \tblam_\ell \in \{\tblam_r\}_{r=1}^R \cap \mathscr{G}_j \right\}.
\]
In words, $\mathscr{L}^j$ is the set of generalized eigenvalues of $\cY$ which are contained in $\mathscr{G}_j$. 

Now suppose $\tblam_m$ and $\tblam_n$ are generalized eigenvalues of $\cY$ such that $\tblam_m \in \mathscr{G}_i$ and $\tblam_n \in \mathscr{G}_\ell$ for some indices $i \neq \ell$. Then there must exist generalized eigenvalues $\blam_m$ and $\blam_n$ of $\cT$ such that $\blam_m \in \mathscr{G}_i$ and $\blam_n \in \mathscr{G}_\ell$ and 
\[
\chi(\blam_m,\tblam_m) \leq \frac{\epsilon}{\sigma_{\min} (\bA) \sigma_{\min} (\bB)} \qquad \text{and} \qquad \chi(\blam_n,\tblam_n) 
\leq \frac{\epsilon}{\sigma_{\min} (\bA) \sigma_{\min} (\bB)}.
\]
Furthermore, from the definition of the $\mathscr{G}_j$ we must have
\[
\chi(\blam_m,\blam_n) \geq 2\delta.
\]
Using the triangle inequality shows
\[
\chi(\tblam_m,\tblam_n) \geq \chi(\blam_m,\blam_n) - \chi(\blam_m,\tblam_m) - \chi(\blam_n,\tblam_n) \geq 2\delta - \frac{2\epsilon}{\sigma_{\min} (\bA) \sigma_{\min} (\bB)}
\]
from which inequality \eqref{eq:GenEigSetGaps} follows. 
\end{proof}

\subsection{Distinct eigenvalue clusters for subpencils of a tensor}
We now use Theorem \ref{theorem:RegPertJEigBound} to give the radius of a Frobenius norm ball around a given tensor in which every tensor is guaranteed to have a subpencil with at least $J$ distinct generalized eigenvalues. 

\begin{corollary}
\label{corollary:ComputedJthBound}
Let $\cT \in \R^{R \times R \times K}$ be a slice mix invertible tensor of real rank $R$. Let $\bU \in \R^{K \times K}$ be any orthogonal matrix. Set $\cS = \cT \cdot_3 \bU$. For each $k=1,\dots, \lfloor K/2 \rfloor$, let $\epsilon_k \geq 0$ be the bound obtained from Theorem \ref{theorem:RegPertJEigBound} which guarantees that matrix pencils in a ball of Frobenius radius $\epsilon_k$ around the subpencil
\[
(\bS_{2k-1},\bS_{2k})
\]
have at least $J$ distinct generalized eigenvalues. Here we take $\epsilon_k=0$ if the corresponding subpencil is not slice mix invertible. Set $\epsilon = ||(\epsilon_1, \dots, \epsilon_{ \lfloor K/2 \rfloor})||_2$. Then every tensor in  an open ball of Frobenius radius $\epsilon$ around $\cT$ has a subpencil which has at least $J$ distinct generalized eigenvalues. 
\end{corollary}
\begin{proof}
Let $\cY \in \R^{R \times R \times K}$ be a tensor such that 
$\|\cT-\cY\|_\text{F} < \epsilon$
and set $\cP=\cY \cdot_3 \bU$. Then we have $\|\cS-\cP\|_\text{F}<\epsilon$. It follows that there must be some index $k \in \{1,\dots, \lfloor K/2 \rfloor\}$ such that 
\[
\spnorm{(\bS_{2k-1},\bS_{2k})-(\bP_{2k-1},\bP_{2k})}  \leq \| (\bS_{2k-1},\bS_{2k})-(\bP_{2k-1},\bP_{2k})\|_\text{F} < \epsilon_k,
\]
hence Theorem \ref{theorem:RegPertJEigBound} may be applied to the subpencil $(\bP_{2k-1},\bP_{2k})$ of $\cP$. In the above, the norm $\spnorm{(\bM_{1},\bM_{2})}$ is the tensor spectral norm of the tensor $\cM \in \R^{R \times R \times 2}$ whose frontal slices are $\bM_{1}$ and $\bM_{2}$.
\end{proof}

Theorem \ref{theorem:RegPertJEigBound} and Corollary \ref{corollary:ComputedJthBound} are similar to results appearing in \cite{EvertD19}; however, the results appearing here are more general and their proofs require additional consideration.

\subsection{Computation of the bounds appearing in Theorem \ref{theorem:RegPertJEigBound} and Corollary \ref{corollary:ComputedJthBound}}

\label{sec:ComputationOfBounds}

We now explain how to compute the bounds presented in this section. We begin by showing how to compute a lower bound for the quantity $\epsilon_2$ in Theorem \ref{theorem:RegPertJEigBound}.

\begin{lemma}
\label{lemma:Eps2lower}
Let $\cT \in \R^{R \times R \times K}$ and let $\cS$ be a tensor which is not slice mix invertible. Then 
\beq
\label{eq:Eps2lower}
\spnorm{\cT-\cS} \geq \max_{\substack{\bv \in \R^K\\ \|\bv\|_2=1}} \sigma_{\min} (\cT \cdot_3 \bv^\mathrm{T}).
\eeq 
\end{lemma}
\begin{proof}

Let $\bv_* \in \R^K$ be a unit vector which maximizes the right hand side of equation \eqref{eq:Eps2lower}. Since $\cS$ is not slice mix invertible, the matrix $\cS \cdot_3 \bv_*^\teT \in \R^{R \times R}$ has a nontrivial null space. It follows that there is a unit vector $\by_* \in \R^R$ such that $\cS \cdot_2 \by_*^\teT \cdot_3 \bv_*^\teT = \b0$. We then have
\[
\spnorm{\cT-\cS} \geq \|(\cT-\cS) \cdot_2 \by_*^\teT \cdot_3 \bv_*^\teT \|_2 = \|(\cT \cdot_3 \bv_*^\teT) \by_* \|_2 \geq \sigma_{\min} (\cT \cdot_3 \bv_*^{\teT})
\]
as claimed.
\end{proof}

While the quantity on the right hand side of inequality \eqref{eq:Eps2lower} is itself nontrivial to compute, one can for example use the lower bound 
\[
\sigma_{\min} (\cT \cdot_3 \bv^\mathrm{T}) \geq \max \{\sigma_{\min} (\bT_1),\dots,\sigma_{\min} (\bT_k)\}.
\]
We use this lower bound for $\epsilon_2$ in the numerical experiment presented in the upcoming Figure \ref{lowrank} and for almost all pencils observe that the quantity $\epsilon_1$ of Theorem \ref{theorem:RegPertJEigBound} is less than or equal to this lower bound, hence the computed bound is rarely affected in our experiments by using this lower bound. 

To compute the quantity $\epsilon_1$, one need only compute a CPD $\cT = \cpd$ of $\cT$ where the columns of $\bC$ are scaled to have unit norm. As shown in Theorem \ref{theorem:RegPertJEigBound}, the generalized eigenvalues of $\cT$ are then equal to the spans of the columns of the factor matrix $\bC$. A subtlety is that the quantity $\sigma_{\min} (\bA) \sigma_{\min} (\bB)$ depends on the scaling of the columns of $\bA$ and $\bB$. We observe good performance if we scale these factor matrices so that for all $r=1,\dots,R$ the norm of $r$th column of $\bA$ is equal to that of the $r$th column of $\bB$, and this is how we scale the factors in our experiments. Locally optimal scaling can be computed using an ALS approach for example. See the discussion in \cite[Section 4.4.1]{EvertD19} for more details. 

For a given unitary $\bU$, the bound presented in Corollary \ref{corollary:ComputedJthBound} may be computed by applying the preceding discussion to the appropriate subpencils of $\cT \cdot_3 \bU$. Computing a unitary which gives the optimal bound in Corollary \ref{corollary:ComputedJthBound} is a nontrivial problem which is related to the travelling salesman problem. See the discussion in \cite[Section 5.1]{EvertD19}.

In practice we randomly generate unitaries $\bU_1,\dots,\bU_L \in \R^{K \times K}$ for some chosen parameter $L$, then keep the maximum bound obtained by applying Corollary \ref{corollary:ComputedJthBound} to the tensors $\cT \cdot_3 \bU_1,\dots, \cT \cdot_3 \bU_L$. The random unitaries we consider have the form $\bQ \cdot \mathrm{diag}(\mathrm{sign}(\bR))$ where $\bQ$ and $\bR$ are obtained from the $\bQ \bR$ decomposition applied to randomly generated matrices with i.i.d.\ Gaussian entries. Here $\mathrm{sign} (\bM)$ applies the $\mathrm{sign}$ operator to each entry of the matrix $\bM$, while $\mathrm{diag} (\bM)$ is the diagonal matrix whose diagonal entries are equal to those of $\bM$. Note that generating unitaries in this method is equivalent to randomly sampling unitaries from the Haar (uniform) measure on the unitary group \cite{Mezz07}.

\section{GESD vs. GEVD accuracy and timing}
\label{sec:GESDPerformance}

We now compare both the accuracy and timing of GESD to GEVD where the GEVD solution is computed using the first two slices of the MLSVD of the tensor\footnote{ Similar to GESD, the average performance of GEVD benefits by using the first two slices of the MLSVD. However, this strategy is not always successful, see Section \ref{sec:BadMLSVDSlices}.}. To this end, we generate low-rank tensors and add noise with varying signal to noise ratios (SNRs). We then compute the CPD with both methods and compare the computation time and factor matrix error (\texttt{cpderr}) of the results. For a low-rank tensor generated from the factor matrices $\mathbf{A}$, $\mathbf{B}$ and $\mathbf{C}$, the factor matrix error of its estimated CPD $\cp{\hat{\mathbf{A}},\hat{\mathbf{B}},\hat{\mathbf{C}}}$ is defined as $\max\left(\frac{\Vert \mathbf{A} - \hat{\mathbf{A}} \Vert_\text{F}}{\Vert \mathbf{A}\Vert_\text{F}},\frac{\Vert \mathbf{B} - \hat{\mathbf{B}} \Vert_\text{F}}{\Vert \mathbf{B}\Vert_\text{F}},\frac{\Vert \mathbf{C} - \hat{\mathbf{C}} \Vert_\text{F}}{\Vert \mathbf{C}\Vert_\text{F}}\right)$, where columns of $\hat{\mathbf{A}}$, $\hat{\mathbf{B}}$ and $\hat{\mathbf{C}}$ have been optimally permuted and scaled to match the columns of $\mathbf{A}$,$\mathbf{B}$ and $\mathbf{C}$, respectively. The SNR of the tensor $\ten{T}$ and the noise $\ten{N}$ in dB is defined as $20\log_{10}\left(\frac{\Vert\ten{T}\Vert_\text{F}}{\Vert\ten{N}\Vert_\text{F}}\right)$. For each SNR, the median error and computation time are computed over 50 trials unless differently specified. The computation time is the time required to obtain the factor matrices $\hat{\mathbf{A}}$, $\hat{\mathbf{B}}$ and $\hat{\mathbf{C}}$ starting from the noisy tensor $\ten{T}+\ten{N}$. All experiments in this section and Section \ref{sec:CRboundPerformance} are performed on a machine with an Intel Core i7-6820HQ CPU at 2.70GHz and
16GB of RAM using MATLAB  R2020b and Tensorlab 3.0 \cite{VDSBL16}. 

We start with two easier problems where the tensor dimensions are considerably larger than the tensor rank. For a more challenging case, we then move on to tensors with factor matrices that have correlated  columns. Next, we examine the accuracy and speed of the GESD method compared to a state-of-the art optimization method.

We note that as a tuning parameter in GESD one may adjust the  chordal distance between generalized eigenvalues which is sufficiently large for splitting between the corresponding generalized eigenspaces. To a point, the accuracy of GESD may be improved by increasing the requested  chordal distance; however, increasing this chordal distance can decrease the speed of GESD. One could also choose to restrict the number of generalized eigenvalue clusters which GESD is allowed to separate in at each step. The optimal splitting strategy in terms of both timing and accuracy is a topic of further investigation. Furthermore, one could additionally consider the separation between computed generalized eigenspaces as a parameter for deciding if a particular split is acceptable.\footnote{An in-depth discussion of the sensitivity of generalized eigenspace computations to the separation between the corresponding generalized eigenspaces is given in \cite{DK87}. While the bounds discussed therein take into account both eigenvalue and eigenspace conditioning, they are expensive to compute and lead to major increases in runtime if used in GESD. This is of course undesirable if the GESD algorithm is used to quickly obtain a good initialization for an optimization algorithm. \cite{ED21} further explores this topic in the context of joint generalized eigenspaces.}  Unless otherwise stated, in the experiments in this section we use a splitting threshold of 0.2 chordal distance.

\subsection{Tensor dimension exceeds tensor rank}
\label{sec:exp-easy}

The following experiments resemble typical problems in signal processing, where the tensor dimensions are often orders of magnitude larger than the tensor rank. For the first experiment, the factor matrix entries of a $100\times 100\times 100$-tensor of rank $10$ are independently sampled from the standard normal distribution. The results are shown in Figure~\ref{easynormal}. In this setting, the GESD algorithm is marginally more accurate than GEVD, but the computational cost of GESD is about  three times that of GEVD. While GESD is more computationally expensive, the first step in this setting is to compress the starting tensor. The cost of this compression dominates the other parts of the computation, thus using GESD over GEVD leads  to about a 25 percent increase in the total run time of the algorithm. 

For the second experiment, the factor matrices are independently sampled from the uniform distribution on $[0,1]$ instead of the standard normal distribution. This setting is more challenging than the first since the columns of the factor matrices are less well-separated. In Figure~\ref{easyuniform}, we can see the accuracies and computation times of both methods. Compared to the case with normally distributed factor matrix entries, a higher SNR is needed to obtain a reasonable solution, but from an SNR of $0$ dB upwards,  the error of the GESD method is more than an order of magnitude smaller than the error for the GEVD method. That is, using GESD leads to a gain of over $20$ dB in accuracy. As in the previous experiment, the total cost of compression together with GESD is only slightly higher than compression with GEVD since compression is the most computationally expensive step.

\subsection{Factor matrices with correlated columns}
\label{sec:exp-cor}

We now consider the more challenging setting where the tensor of interest has rank equal to dimensions. That is, the tensors we consider in this section have rank $R$ and size $R \times R \times R$. This case is significantly more difficult than the $I_1 \times I_2 \times I_3$ rank $R$ case with $I_1,I_2,I_3\gg R$ which generally occurs in practice and is discussed in Section~\ref{sec:exp-easy}.  Recall that the first step in both GESD and GEVD is to compute a MLSVD of the tensor of interest. In the $I_1,I_2,I_3\gg R$ case, computing this MLSVD and reducing to an $R \times R \times R$ tensor typically significantly increases the SNR of the problem thereby reducing the difficulty of the problem. In addition, the compression generally takes much more time than computing the CPD of the small core tensor using GEVD or GESD.

In Figure \ref{cor} we consider the particularly difficult case where the columns of the factor matrices are highly correlated. More concretely, we consider the case where the first column of $\mathbf{A}$ is sampled from the uniform distribution on $[0,1]$ as before and each successive column of $\mathbf{A}$ is randomly sampled so that it forms a $10^\circ$ angle with $\mathbf{a}_{1}$. The other factor matrices are sampled analogously. Results for this case are shown in Figure~\ref{cor}. We note that, due to the difficulty of this setting, very high SNR is required before either GESD or GEVD is able to compute a meaningful solution.

For the problems in Figure~\ref{cor}, we see that GESD is significantly more accurate than GEVD. In fact, GESD is able to compute meaningful solutions starting at about an order of magnitude lower SNR than GEVD. This demonstrates that the handling of pencil specific eigenvalue clusters is better in GESD than in GEVD due to the fact that GESD is able to compute an eigenspace for the cluster, but delays computation of eigenvectors to a pencil in which the previously clustered eigenvalues are better separated. As before, the GESD algorithm takes about three times as long as the GEVD algorithm. 

In Figure \ref{asymp}, we investigate the speed and accuracy of both algorithms as a function of tensor rank. Here we run GESD and GEVD on noiseless tensors of rank $R$ and size $R \times R \times R$ where $R$ ranges from $50$ to $350$ and where the factor matrices are independently sampled from the uniform distribution on $[0,1]$.  For all values of $R$ we see that GESD is roughly two orders of magnitude more accurate than GEVD, but that GESD takes about three times as long to run as GEVD.

\subsection{Comparison with optimization}

As the GEVD algorithm can quickly compute a reasonably good fit for the CPD, its solution is often used as an initialization for an optimization algorithm. By providing a good initialization, the number of iterations of the more accurate yet more computationally expensive optimization algorithm can ideally be lowered significantly. In the next experiment, we compare the error and computation time of a nonlinear least-squares (NLS) algorithm~\cite{VD19} for three different initialization strategies: starting from the GEVD solution, the GESD solution or a random initial point. The optimization algorithm is run until either the relative cost function change or the relative factor matrix change drops under $10^{-8}$. The time to compute the GEVD or GESD solutions is included in the runtime of the optimization algorithm. 

In Figure~\ref{optimnocorsmall}, median values are shown over $20$ trials for a $10\times 10\times 10$-tensor of rank $10$ with factor matrices independently sampled from the uniform distribution on $[0,1]$. We can see that the NLS algorithm generally converges to a good solution for all three initialization strategies. The GEVD solution is notably less accurate than either a GESD or optimization based solution. In the right plot of Figure~\ref{optimnocorsmall}, it can be seen that the GESD method is about an order of magnitude faster than the optimization algorithm.  Also, the computation time of the optimization algorithm is significantly reduced by starting from the GEVD or GESD initialization, even when their computation time is included. For all SNRs, computation of a GESD or GEVD initialization followed by optimization is equally fast.

In Figure~\ref{optimcorsmall},  median values are shown over $20$ trials for a $10\times 10\times 10$-tensor of rank $10$ where the columns of the factor matrices are sampled with a fixed $10^\circ$ angle between them. For this more challenging problem, we can see from the left plot of Figure~\ref{optimcorsmall} that decent algebraic initializations are only achieved at higher SNR values compared to the uncorrelated case. Additionally we see that the SNR required for GESD to compute reasonably accurate initializations is significantly lower than that for GEVD. In this experiment, random initialization of the optimization algorithm does not lead to a meaningful solution at all, which again highlights the importance of algebraic initialization. As before, optimization with a GESD initialization and optimization with a GEVD initialization typically take the same amount of time.

In Figure~\ref{optimcorsmall20} we consider an even more difficult setting. Here median values are shown over $100$ trials for a $20\times 20\times 20$-tensor of rank $20$ where the columns of the factor matrices are sampled with a fixed $8^\circ$ angle between them. Due to the difficulty of this setting, meaningful solutions can only be computed with very high SNR. In the left plot of Figure~\ref{optimcorsmall20} we see both that GESD computes a more accurate algebraic initialization than GEVD and, moreover, that using the algebraic initialization from GESD for optimization leads to a more accurate CPD than using the algebraic initialization from GEVD. In addition we observe that GESD+optimization is faster than GEVD+optimization. As in the previous figure, optimization with random initialization does not lead to a meaningful solution. We emphasize that in Figure ~\ref{optimcorsmall20}, optimization following either GESD or GEVD ultimately stops due to insufficient advancement of the objective function. This illustrates that, in this problem, using GESD based initialization  indeed leads to better local minima than GEVD based initialization.

These experiments illustrate that using GESD as an algebraic initialization compared to GEVD does not cause any loss in terms of time or accuracy. In fact, GESD+optimization can be both more accurate and faster than GEVD+optimization. Furthermore see we that GESD based initializations are significantly more accurate than GEVD based initializations.  Thus using GESD over GEVD to initialize gives more reliable initializations without impacting total computation time.

\begin{figure}[h!]
	\centering    
%
%
\definecolor{mycolor1}{rgb}{0.00000,0.44700,0.74100}%
\definecolor{mycolor2}{rgb}{0.85000,0.32500,0.09800}%
\definecolor{mycolor3}{rgb}{0.92900,0.69400,0.12500}%
\begin{tikzpicture}

\begin{axis}[%
width=4.5cm,
height=2.5cm,
at={(0cm,0cm)},
scale only axis,
axis lines = left,
xmin=-22,
xmax=42,
xtick={-20,0,20,40},
xlabel={SNR (dB)},
ylabel={\texttt{cpderr}},
ytick={0.00001,0.0001,0.001,0.01,0.1,1},
ymode=log,
ymin=1e-4,
ymax=1,
yminorticks=true,
axis background/.style={fill=white},
legend style={legend cell align=left,align=left,draw=white!15!black}
]
\addplot [color=mycolor1,solid,line width=2.0pt,mark size=3.0pt,mark=asterisk,mark options={solid}]
  table[row sep=crcr]{%
-20	0.874545009195119\\
-15	0.468872131736895\\
-10	0.316198892949336\\
-5	0.118263371257622\\
0	0.0475584113618507\\
5	0.0251538705086156\\
10	0.0147488243530038\\
15	0.00743862775454093\\
20	0.00422722554914452\\
25	0.00243234676125892\\
30	0.00139195047434574\\
35	0.00074944533193058\\
40	0.000404860733943089\\
};
\addplot [color=mycolor2,solid,line width=2.0pt,mark size=3.0pt,mark=o,mark options={solid}]
  table[row sep=crcr]{%
-20	0.891529316961345\\
-15	0.488572713326469\\
-10	0.151676743804593\\
-5	0.0688856727050255\\
0	0.0353899901319866\\
5	0.0193341841931293\\
10	0.010772805530983\\
15	0.00607792792478353\\
20	0.00341144299260642\\
25	0.00191726771863074\\
30	0.00107511458739127\\
35	0.000602865392668655\\
40	0.000341568511297233\\
};
\node [mycolor1] at (470,-4){\textbf{GEVD}};
\node [mycolor2] at (420,-8){\textbf{GESD}};
\end{axis}
\begin{axis}[%
width=4.5cm,
height=2.5cm,
at={(6.5cm,0cm)},
scale only axis,
axis lines = left,
xmin=-22,
xmax=42,
xtick={-20,0,20,40},
xlabel={SNR (dB)},
ymin=0,
ymax=0.04,
ylabel={Time (s)},
yticklabel style={
	/pgf/number format/fixed,
	/pgf/number format/precision=4
},
scaled y ticks=false,
axis background/.style={fill=white},
legend style={legend cell align=left,align=left,draw=white!15!black}
]
\addplot [color=mycolor1,solid,line width=2.0pt,mark size=3.0pt,mark=asterisk,mark options={solid}]
  table[row sep=crcr]{%
-20	0.0047378\\
-15	0.0047916\\
-10	0.0047015\\
-5	0.00472005\\
0	0.0046985\\
5	0.00468405\\
10	0.00479225\\
15	0.0046099\\
20	0.00463505\\
25	0.0046135\\
30	0.00464615\\
35	0.0046483\\
40	0.0046224\\
};
\addplot [color=mycolor2,solid,line width=2.0pt,mark size=3.0pt,mark=o,mark options={solid}]
  table[row sep=crcr]{%
-20	0.0106587\\
-15	0.01224005\\
-10	0.0135287\\
-5	0.0136144\\
0	0.01347825\\
5	0.01413385\\
10	0.01370365\\
15	0.0134571\\
20	0.01337645\\
25	0.01354055\\
30	0.01362445\\
35	0.0136032\\
40	0.01353225\\
};
\addplot [color=mycolor3,solid,line width=2.0pt,mark size=3.0pt,mark=square,mark options={solid}]
  table[row sep=crcr]{%
-20	0.0271272\\
-15	0.0263936\\
-10	0.0264121\\
-5	0.0266516\\
0	0.02677975\\
5	0.0265397\\
10	0.0270025\\
15	0.0266301\\
20	0.0268298\\
25	0.02677475\\
30	0.0269448\\
35	0.0269011\\
40	0.02695685\\
};
\node [mycolor1] at (250,84){\textbf{GEVD}};
\node [mycolor2] at (250,200){\textbf{GESD}};
\node [mycolor3] at (250,330){\textbf{Compression}};
\end{axis}
\end{tikzpicture}%
	\caption{For normally distributed factor matrices, the  GESD and GEVD methods have similar performance with respect to the factor matrix accuracy. Median (over $50$ trials) \texttt{cpderr} (left) and computation time (right) for rank $10$ tensors of dimensions $100\times 100\times 100$ with factor matrices sampled from the standard normal distribution. The runtime of  GESD on the core tensor is about three times that of GEVD on the core; however overall runtime is dominated by the compression step.}
	\label{easynormal}
\end{figure}
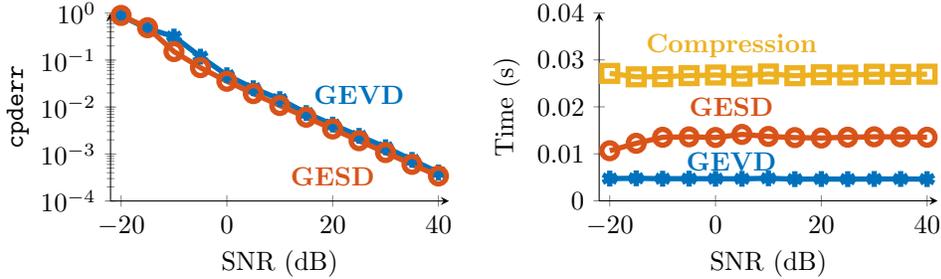

\begin{figure}[h!]
	\centering 
%
%
\definecolor{mycolor1}{rgb}{0.00000,0.44700,0.74100}%
\definecolor{mycolor2}{rgb}{0.85000,0.32500,0.09800}%
\definecolor{mycolor3}{rgb}{0.92900,0.69400,0.12500}%
\begin{tikzpicture}

\begin{axis}[%
width=4.5cm,
height=2.5cm,
at={(0cm,0cm)},
scale only axis,
axis lines = left,
xmin=-22,
xmax=42,
xtick={-20,0,20,40},
xlabel={SNR (dB)},
ylabel={\texttt{cpderr}},
ytick={0.00001,0.0001,0.001,0.01,0.1,1},
ymode=log,
ymin=1e-3,
ymax=1,
yminorticks=true,
axis background/.style={fill=white},
legend style={legend cell align=left,align=left,draw=white!15!black}
]
\addplot [color=mycolor1,solid,line width=2.0pt,mark size=3.0pt,mark=asterisk,mark options={solid}]
  table[row sep=crcr]{%
-20	0.951884617212015\\
-15	0.946311138344353\\
-10	0.942668387215704\\
-5	0.782640637869836\\
0	0.54803085122734\\
5	0.405731082494725\\
10	0.284775100975207\\
15	0.20676761663919\\
20	0.091202971356516\\
25	0.072784945157129\\
30	0.03609498856085\\
35	0.0188699213494627\\
40	0.0150309338911144\\
};
\addplot [color=mycolor2,solid,line width=2.0pt,mark size=3.0pt,mark=o,mark options={solid}]
  table[row sep=crcr]{%
-20	0.949120758537804\\
-15	0.945897469970095\\
-10	0.936766932302588\\
-5	0.776310666066095\\
0	0.485992378967873\\
5	0.128783990907458\\
10	0.0642893451427937\\
15	0.0334221845643899\\
20	0.0192065749268487\\
25	0.0103482408547231\\
30	0.00608537381526533\\
35	0.00338635284909589\\
40	0.00186977204728153\\
};
\node [mycolor1] at (470,-1){\textbf{GEVD}};
\node [mycolor2] at (420,-5.5){\textbf{GESD}};
\end{axis}
\begin{axis}[%
width=4.5cm,
height=2.5cm,
at={(6.5cm,0cm)},
scale only axis,
axis lines = left,
xmin=-22,
xmax=42,
xtick={-20,0,20,40},
xlabel={SNR (dB)},
ymin=0,
ymax=0.04,
ylabel={Time (s)},
yticklabel style={
	/pgf/number format/fixed,
	/pgf/number format/precision=4
},
scaled y ticks=false,
axis background/.style={fill=white},
legend style={legend cell align=left,align=left,draw=white!15!black}
]
\addplot [color=mycolor1,solid,line width=2.0pt,mark size=3.0pt,mark=asterisk,mark options={solid}]
  table[row sep=crcr]{%
-20	0.0047427\\
-15	0.0046742\\
-10	0.0049032\\
-5	0.00471865\\
0	0.0047399\\
5	0.00476895\\
10	0.00469105\\
15	0.0047249\\
20	0.00471265\\
25	0.0046415\\
30	0.00467695\\
35	0.0046991\\
40	0.0046824\\
};
\addplot [color=mycolor2,solid,line width=2.0pt,mark size=3.0pt,mark=o,mark options={solid}]
  table[row sep=crcr]{%
-20	0.0109119\\
-15	0.01066515\\
-10	0.01075875\\
-5	0.01190615\\
0	0.01550155\\
5	0.0165725\\
10	0.0169642\\
15	0.01746855\\
20	0.0169497\\
25	0.0170365\\
30	0.0169419\\
35	0.01690405\\
40	0.01721255\\
};
\addplot [color=mycolor3,solid,line width=2.0pt,mark size=3.0pt,mark=square,mark options={solid}]
  table[row sep=crcr]{%
-20	0.02873935\\
-15	0.0267441\\
-10	0.02731775\\
-5	0.02688335\\
0	0.0281202\\
5	0.02702965\\
10	0.026826\\
15	0.0270413\\
20	0.0282226\\
25	0.02736915\\
30	0.0267166\\
35	0.0267728\\
40	0.02674355\\
};
\node [mycolor1] at (260,81){\textbf{GEVD}};
\node [mycolor2] at (260,212){\textbf{GESD}};
\node [mycolor3] at (260,350){\textbf{Compression}};
\end{axis}
\end{tikzpicture}%
	\caption{For uniformly distributed factor matrices, the GESD method obtains a higher accuracy for the estimated factor matrices than the GEVD method. Median (over $50$ trials) \texttt{cpderr} (left) and computation time (right) for rank $10$ tensors of dimensions $100\times 100\times 100$ with factor matrices sampled from the uniform distribution on $[0,1]$.}
	\label{easyuniform}
\end{figure}
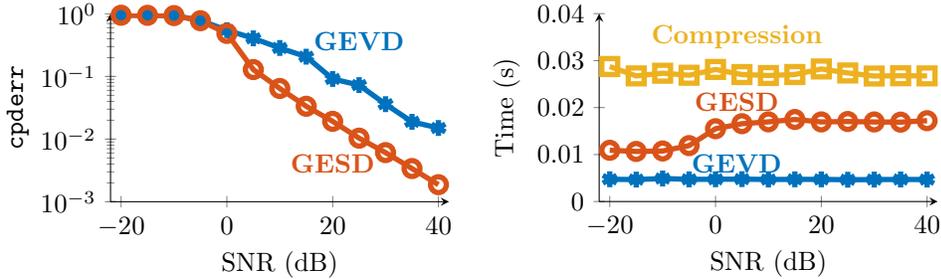

\begin{figure}
	\centering   
%
%
\definecolor{mycolor1}{rgb}{0.00000,0.44700,0.74100}%
\definecolor{mycolor2}{rgb}{0.85000,0.32500,0.09800}%
\begin{tikzpicture}

\begin{axis}[%
width=4.5cm,
height=2.5cm,
at={(0cm,0cm)},
scale only axis,
axis lines = left,
xmin=58,
xmax=122,
xlabel={SNR (dB)},
ylabel={\texttt{cpderr}},
ymode=log,
ymin=2e-03,
ymax=1,
ytick={0.0001,0.001,0.01,0.1,1},
yminorticks=true,
axis background/.style={fill=white},
legend style={legend cell align=left,align=left,draw=white!15!black}
]
\addplot [color=mycolor1,solid,line width=2.0pt,mark size=3.0pt,mark=asterisk,mark options={solid}]
  table[row sep=crcr]{%
40	0.725726427239562\\
45	0.664320462991191\\
50	0.599013817156048\\
55	0.52791318317128\\
60	0.470359490693395\\
65	0.456861974252238\\
70	0.400394652287972\\
75	0.389356636161652\\
80	0.351129367967178\\
85	0.349814527427354\\
90	0.32264237714701\\
95	0.27321863733608\\
100	0.196468201789664\\
105	0.122898151265776\\
110	0.101471828069069\\
115	0.0318934354661484\\
120	0.015251093158543\\
};

\addplot [color=mycolor2,solid,line width=2.0pt,mark size=3.0pt,mark=o,mark options={solid}]
  table[row sep=crcr]{%
40	0.733893569807443\\
45	0.651810101440347\\
50	0.585941445510382\\
55	0.514693400489467\\
60	0.446942202211805\\
65	0.412490793656016\\
70	0.368473084373307\\
75	0.349561103588612\\
80	0.331756995593841\\
85	0.315675762498562\\
90	0.241885176029033\\
95	0.129412376971729\\
100	0.0387248666039223\\
105	0.0206170702617332\\
110	0.0119658287792394\\
115	0.00684590016547395\\
120	0.00296221249201146\\
};
\node [mycolor1] at (650,-0.5){\textbf{GEVD}};
\node [mycolor2] at (550,-4.5){\textbf{GESD}};
\end{axis}
\begin{axis}[%
width=4.5cm,
height=2.5cm,
at={(6.5cm,0cm)},
scale only axis,
axis lines = left,
xmin=58,
xmax=122,
xlabel={SNR (dB)},
ymin=0.001,
ymode = log,
ymax=0.1,
ylabel={Time (s)},
axis background/.style={fill=white},
legend style={legend cell align=left,align=left,draw=white!15!black}
]
\addplot [color=mycolor1,solid,line width=2.0pt,mark size=3.0pt,mark=asterisk,mark options={solid}]
  table[row sep=crcr]{%
40	0.0054697\\
45	0.00492835\\
50	0.00493695\\
55	0.0048623\\
60	0.0048719\\
65	0.0049409\\
70	0.00489465\\
75	0.00497585\\
80	0.0050159\\
85	0.00488985\\
90	0.00491445\\
95	0.0049181\\
100	0.0047721\\
105	0.0048123\\
110	0.00491805\\
115	0.00477735\\
120	0.00483375\\
};

\addplot [color=mycolor2,solid,line width=2.0pt,mark size=3.0pt,mark=o,mark options={solid}]
  table[row sep=crcr]{%
40	0.01182355\\
45	0.01230345\\
50	0.01298225\\
55	0.01535075\\
60	0.0173893\\
65	0.0174805\\
70	0.0181095\\
75	0.01930355\\
80	0.0184852\\
85	0.0192056\\
90	0.02089\\
95	0.02019905\\
100	0.0203673\\
105	0.0204298\\
110	0.0199924\\
115	0.0199972\\
120	0.02037025\\
};
\node [mycolor1] at (280,-4.8){\textbf{GEVD}};
\node [mycolor2] at (280,-3.4){\textbf{GESD}};
\end{axis}

\end{tikzpicture}%
	\caption{For factor matrices with correlated columns, the GESD method performs significantly better than the GEVD method with respect to the factor accuracy. Median (over $50$ trials) \texttt{cpderr} (left) and computation time (right) for rank $10$ tensors of dimensions $10\times 10\times 10$ with factor matrices with correlated columns ($10^\circ$ angles). Because of the difficulty of the problem only very high SNR values lead to a decent solution. The computation time of the GESD method is about  three times that of the GEVD method.}
	\label{cor}
\end{figure}
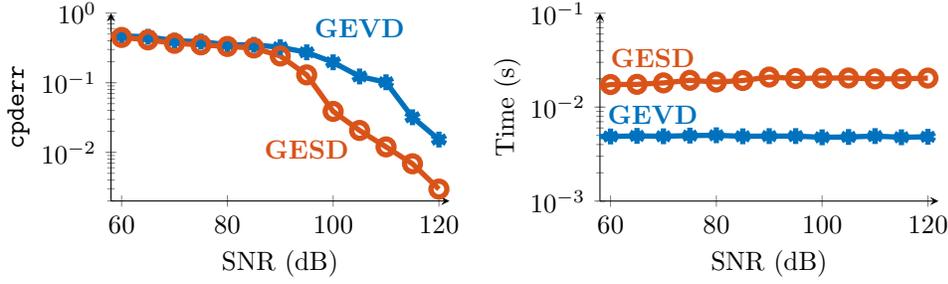
\begin{figure}
	\centering 
%
%
\definecolor{mycolor1}{rgb}{0.00000,0.44700,0.74100}%
\definecolor{mycolor2}{rgb}{0.85000,0.32500,0.09800}%
\begin{tikzpicture}

\begin{axis}[%
width=4.5cm,
height=2.5cm,
at={(0cm,0cm)},
scale only axis,
axis lines = left,
xmin=0,
xmax=360,
xlabel={$R$},
ymode=log,
ymin=1e-13,
ymax=1e-8,
yminorticks=true,
ylabel={\texttt{cpderr}},
axis background/.style={fill=white},
legend style={legend cell align=left,align=left,draw=white!15!black}
]
\addplot [color=mycolor1,solid,line width=2.0pt,mark size=3.0pt,mark=asterisk,mark options={solid}]
  table[row sep=crcr]{%
50	1.40753517824462e-11\\
100	1.1050170710044e-10\\
150	2.29678012529945e-10\\
200	1.47286831580077e-09\\
250	3.19890857280419e-09\\
300	2.31619558508833e-09\\
350	5.49198867810111e-09\\
};

\addplot [color=mycolor2,solid,line width=2.0pt,mark size=3.0pt,mark=o,mark options={solid}]
  table[row sep=crcr]{%
50	4.45083893544801e-13\\
100	1.5226500732203e-12\\
150	1.5233917446853e-12\\
200	6.87435112176926e-12\\
250	1.0081430972777e-11\\
300	7.72254955152168e-12\\
350	1.66623246521134e-11\\
};
\node [mycolor1] at (190,-19){\textbf{GEVD}};
\node [mycolor2] at (190,-24){\textbf{GESD}};
\end{axis}

\begin{axis}[%
width=4.5cm,
height=2.5cm,
at={(6.5cm,0cm)},
scale only axis,
axis lines = left,
xmin=0,
xmax=360,
xlabel={$R$},
ymode=log,
ymin=0.01,
ymax=100,
yminorticks=true,
ylabel={Time (s)},
axis background/.style={fill=white},
legend style={at={(0.786,0.446)},anchor=south west,legend cell align=left,align=left,draw=white!15!black}
]
\addplot [color=mycolor1,solid,line width=2.0pt,mark size=3.0pt,mark=asterisk,mark options={solid}]
  table[row sep=crcr]{%
50	0.06083805\\
100	0.32284025\\
150	1.1090784\\
200	2.9568713\\
250	6.41655225\\
300	11.9852996\\
350	20.466918\\
};

\addplot [color=mycolor2,solid,line width=2.0pt,mark size=3.0pt,mark=o,mark options={solid}]
  table[row sep=crcr]{%
50	0.1527943\\
100	0.7690414\\
150	2.965329\\
200	8.14552195\\
250	17.68447725\\
300	35.47537805\\
350	67.4404437\\
};
\node [mycolor1] at (150,-2){\textbf{GEVD}};
\node [mycolor2] at (80,1.5){\textbf{GESD}};
\end{axis}
\end{tikzpicture}%
	\caption{The relative time difference between the GESD and GEVD methods  remains about the same as tensor rank increases. Median (over $20$ trials) \texttt{cpderr} and computation time for exact rank $R$ tensors of dimensions $R\times R\times R$ with factor matrices sampled from the uniform distribution on $[0,1]$. The GESD method obtains factors with a higher accuracy than the GEVD method for all tensor dimensions.  In this experiment we use a splitting threshold of $5/R$.}
	\label{asymp}
\end{figure}
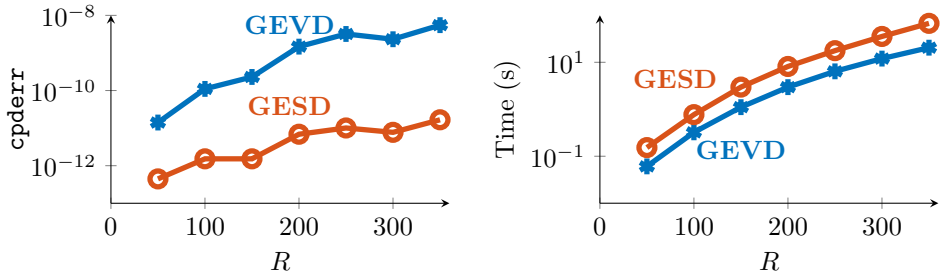

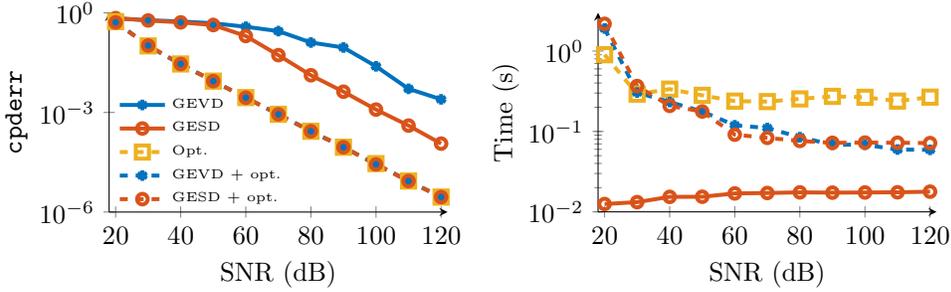
\begin{figure}
	\centering 
%
%
\definecolor{mycolor1}{rgb}{0.00000,0.44700,0.74100}%
\definecolor{mycolor2}{rgb}{0.85000,0.32500,0.09800}%
\definecolor{mycolor3}{rgb}{0.92900,0.69400,0.12500}%
\definecolor{mycolor4}{rgb}{0.49400,0.18400,0.55600}%
\definecolor{mycolor5}{rgb}{0.46600,0.67400,0.18800}%
\begin{tikzpicture}

\begin{axis}[%
width=4.5cm,
height=2.65cm,
at={(0cm,0cm)},
scale only axis,
axis lines = left,
xmin=18,
xmax=122,
xlabel={SNR (dB)},
ymode=log,
ymin=1e-6,
ymax=1,
yminorticks=true,
ylabel={\texttt{cpderr}},
axis background/.style={fill=white},
legend style={legend cell align=left,align=left,draw=none, at={(0.55,0.625)},fill=none}
]
\addplot [color=mycolor1,solid,line width=1.5pt,mark size=2.0pt,mark=asterisk,mark options={solid}]
  table[row sep=crcr]{%
20	0.683203503038114\\
30	0.599607523253934\\
40	0.5493666056244\\
50	0.485370030086641\\
60	0.380965875053933\\
70	0.281908641019267\\
80	0.127807670843945\\
90	0.090425768293556\\
100	0.0244986961455524\\
110	0.0051552023615498\\
120	0.00248311801559746\\
};
\addlegendentry{\tiny{GEVD}};
\addplot [color=mycolor2,solid,line width=1.5pt,mark size=2.0pt,mark=o,mark options={solid}]
  table[row sep=crcr]{%
20	0.695343658128347\\
30	0.589433634203887\\
40	0.516492168173632\\
50	0.42865113439504\\
60	0.200438529858601\\
70	0.0532370883360625\\
80	0.0132014525120288\\
90	0.00421794132663423\\
100	0.00122064939573851\\
110	0.000401769481546915\\
120	0.000115132848141501\\
};
\addlegendentry{\tiny{GESD}};
\addplot [color=mycolor3, dashed,line width=1.5pt,mark size=2.5pt,mark=square,mark options={solid}]
  table[row sep=crcr]{%
20	0.542216577254006\\
30	0.101100957788953\\
40	0.0288628576648067\\
50	0.00878500536697244\\
60	0.00280629464903993\\
70	0.00087956145627145\\
80	0.000271376321123935\\
90	8.95741936951077e-05\\
100	2.8784802355337e-05\\
110	8.57600735850697e-06\\
120	2.82714414808915e-06\\
};
\addlegendentry{\tiny{Opt.}};
\addplot [color=mycolor1,dashed,line width=1.5pt,mark size=2.0pt,mark=asterisk,mark options={solid}]
  table[row sep=crcr]{%
20	0.545132296197856\\
30	0.100600448094018\\
40	0.0282011385623642\\
50	0.00878531360866486\\
60	0.00278820664921864\\
70	0.000879559915310117\\
80	0.000271387329312748\\
90	8.9520654071008e-05\\
100	2.7236838536433e-05\\
110	8.57600034499342e-06\\
120	2.82541008264594e-06\\
};
\addlegendentry{\tiny{GEVD + opt.}};
\addplot [color=mycolor2,dashed,line width=1.5pt,mark size=2.0pt,mark=o,mark options={solid}]
  table[row sep=crcr]{%
20	0.515604710107614\\
30	0.103553203589616\\
40	0.0288645895946101\\
50	0.00878516472079114\\
60	0.0027882146965142\\
70	0.000879559839137909\\
80	0.000271387377490958\\
90	8.95206711954753e-05\\
100	2.72368426881819e-05\\
110	8.57599755604329e-06\\
120	2.82540909842727e-06\\
};
\addlegendentry{\tiny{GESD + opt.}};
\end{axis}

\begin{axis}[%
width=4.5cm,
height=2.65cm,
at={(6.5cm,0cm)},
scale only axis,
axis lines = left,
xmin=18,
xmax=122,
yticklabel style={
	/pgf/number format/fixed,
	/pgf/number format/precision=2
},
xlabel={SNR (dB)},
ymin=1e-2,
ymax=3,
ymode=log,
ylabel={Time (s)},
axis background/.style={fill=white},
legend style={legend cell align=left,align=left,draw=white!15!black}
]
%
%
\addplot [color=mycolor2,solid,line width=1.5pt,mark size=2.0pt,mark=o,mark options={solid}]
  table[row sep=crcr]{%
20	0.0125045\\
30	0.0132127\\
40	0.01536235\\
50	0.0154113\\
60	0.01694725\\
70	0.017162\\
80	0.0174707\\
90	0.01733325\\
100	0.0174122\\
110	0.01750865\\
120	0.01782655\\
};

\addplot [color=mycolor3, dashed,line width=1.5pt,mark size=2.5pt,mark=square,mark options={solid}]
  table[row sep=crcr]{%
20	0.8996942\\
30	0.29106705\\
40	0.336009\\
50	0.2814288\\
60	0.23843985\\
70	0.2352689\\
80	0.25385125\\
90	0.27319585\\
100	0.2660934\\
110	0.23871285\\
120	0.26719085\\
};

\addplot [color=mycolor1,dashed,line width=1.5pt,mark size=2.0pt,mark=asterisk,mark options={solid}]
  table[row sep=crcr]{%
20	1.9052624\\
30	0.3050582\\
40	0.2322806\\
50	0.17726065\\
60	0.11936595\\
70	0.1102767\\
80	0.08482175\\
90	0.0684884\\
100	0.06827015\\
110	0.0595927\\
120	0.05900775\\
};

\addplot [color=mycolor2,dashed,line width=1.5pt,mark size=2.0pt,mark=o,mark options={solid}]
  table[row sep=crcr]{%
20	2.1788289\\
30	0.36714725\\
40	0.2080486\\
50	0.1770741\\
60	0.0912355\\
70	0.0832162\\
80	0.07623565\\
90	0.072202\\
100	0.07268225\\
110	0.07239295\\
120	0.07164205\\
};
\end{axis}
\end{tikzpicture}%
	\caption{Initializing an optimization algorithm with the GESD or GEVD solution is faster than starting from a random initialization. Median (over $20$ trials) \texttt{cpderr} (left) and computation time (right) for rank $10$ tensors of dimensions $10\times 10\times 10$ with factor matrices sampled from the uniform distribution on $[0,1]$. All initializations lead to the same solution. Optimization after initializing with the more accurate GESD solution is equally fast as optimization after initializing with the GEVD solution.}
	\label{optimnocorsmall}
\end{figure}
\begin{figure}
	\centering   
%
%
\definecolor{mycolor1}{rgb}{0.00000,0.44700,0.74100}%
\definecolor{mycolor2}{rgb}{0.85000,0.32500,0.09800}%
\definecolor{mycolor3}{rgb}{0.92900,0.69400,0.12500}%
\definecolor{mycolor4}{rgb}{0.49400,0.18400,0.55600}%
\definecolor{mycolor5}{rgb}{0.46600,0.67400,0.18800}%
\begin{tikzpicture}

\begin{axis}[%
width=4.5cm,
height=2.65cm,
at={(0cm,0cm)},
scale only axis,
axis lines = left,
xmin=18,
xmax=122,
xlabel={SNR (dB)},
ymode=log,
ymin=1e-5,
ymax=1,
yminorticks=true,
ylabel={\texttt{cpderr}},
axis background/.style={fill=white},
legend style={legend cell align=left,align=left,draw=none, at={(0.55,0.625)},fill=none}
]
\addplot [color=mycolor1,solid,line width=1.5pt,mark size=2.0pt,mark=asterisk,mark options={solid}]
  table[row sep=crcr]{%
20	0.88006115378495\\
30	0.870384183828502\\
40	0.731466557126747\\
50	0.587483650915691\\
60	0.467791774058721\\
70	0.406977627745765\\
80	0.360811430027644\\
90	0.290108004977683\\
100	0.242777053650975\\
110	0.213860757196203\\
120	0.0444626835931418\\
};
\addlegendentry{\tiny{GEVD}};
\addplot [color=mycolor2,solid,line width=1.5pt,mark size=2.0pt,mark=o,mark options={solid}]
  table[row sep=crcr]{%
20	0.897834007949406\\
30	0.861558458690532\\
40	0.72523817225309\\
50	0.56727185120639\\
60	0.474594641099878\\
70	0.3698883746823\\
80	0.32618429080653\\
90	0.304748113253124\\
100	0.0224797554953618\\
110	0.0109636674492754\\
120	0.003689191961116\\
};
\addlegendentry{\tiny{GESD}};
\addplot [color=mycolor3, dashed,line width=1.5pt,mark size=2.5pt,mark=square,mark options={solid}]
  table[row sep=crcr]{%
20	0.851066924123612\\
30	0.808747639501973\\
40	0.593214866146594\\
50	0.346096412241622\\
60	0.139711648899097\\
70	0.151852799437846\\
80	0.0971834434083028\\
90	0.182812083981997\\
100	0.156295808578158\\
110	0.0957428355303514\\
120	0.109242877295055\\
};
\addlegendentry{\tiny{Opt.}};
\addplot [color=mycolor1,dashed,line width=1.5pt,mark size=2.0pt,mark=asterisk,mark options={solid}]
  table[row sep=crcr]{%
20	0.909318140617458\\
30	0.857361802357413\\
40	0.625270530075553\\
50	0.347690971807495\\
60	0.052074923472418\\
70	0.0129633454887122\\
80	0.00384311332507874\\
90	0.00126050110728903\\
100	0.000368073663883311\\
110	0.000120238212684844\\
120	3.70614743962261e-05\\
};
\addlegendentry{\tiny{GEVD + opt.}};
\addplot [color=mycolor2,dashed,line width=1.5pt,mark size=2.0pt,mark=o,mark options={solid}]
  table[row sep=crcr]{%
20	0.902630399638805\\
30	0.837935781703851\\
40	0.601112907201379\\
50	0.35753476445843\\
60	0.0559896742543946\\
70	0.0124507019157922\\
80	0.00384443809446163\\
90	0.00125834678924929\\
100	0.00036324688352551\\
110	0.000120518871880476\\
120	3.73248330405573e-05\\
};
\addlegendentry{\tiny{GESD + opt.}};
\end{axis}

\begin{axis}[%
width=4.5cm,
height=2.65cm,
at={(6.5cm,0cm)},
scale only axis,
axis lines = left,
xmin=18,
xmax=122,
yticklabel style={
	/pgf/number format/fixed,
	/pgf/number format/precision=2
},
xlabel={SNR (dB)},
ymin=1e-1,
ymax=10,
ymode=log,
ylabel={Time (s)},
axis background/.style={fill=white},
legend style={legend cell align=left,align=left, at={(0.55,0.75)},draw=white!15!black}
]
\addplot [color=mycolor1,solid,line width=1.5pt,mark size=2.0pt,mark=asterisk,mark options={solid}]
  table[row sep=crcr]{%
20	0.00554285\\
30	0.0054847\\
40	0.0052236\\
50	0.0053294\\
60	0.0054027\\
70	0.0050748\\
80	0.00492985\\
90	0.0049883\\
100	0.0053143\\
110	0.00538545\\
120	0.00524065\\
};

\addplot [color=mycolor2,solid,line width=1.5pt,mark size=2.0pt,mark=o,mark options={solid}]
  table[row sep=crcr]{%
20	0.01291595\\
30	0.01252515\\
40	0.01296435\\
50	0.0143519\\
60	0.0177518\\
70	0.0180406\\
80	0.0196594\\
90	0.0216384\\
100	0.02247705\\
110	0.0211019\\
120	0.02151415\\
};

\addplot [color=mycolor3, dashed,line width=1.5pt,mark size=2.5pt,mark=square,mark options={solid}]
  table[row sep=crcr]{%
20	0.72799855\\
30	0.64043115\\
40	0.75201775\\
50	0.6242321\\
60	0.63961855\\
70	0.70309735\\
80	0.6328114\\
90	0.55358085\\
100	0.7277947\\
110	0.74279545\\
120	0.6344215\\
};

\addplot [color=mycolor1,dashed,line width=1.5pt,mark size=2.0pt,mark=asterisk,mark options={solid}]
  table[row sep=crcr]{%
20	3.4998204\\
30	3.5405405\\
40	3.20238345\\
50	4.0895378\\
60	1.9905068\\
70	1.1255846\\
80	0.84559685\\
90	0.49309725\\
100	0.42253385\\
110	0.2891481\\
120	0.2361133\\
};

\addplot [color=mycolor2,dashed,line width=1.5pt,mark size=2.0pt,mark=o,mark options={solid}]
  table[row sep=crcr]{%
20	3.112642\\
30	4.78493315\\
40	4.18884005\\
50	3.37539825\\
60	2.18001625\\
70	1.1110387\\
80	0.609472\\
90	0.59127295\\
100	0.2523524\\
110	0.2132054\\
120	0.18751655\\
};
\end{axis}
\end{tikzpicture}%
	\caption{For factor matrices with correlated columns, the GESD solution is more accurate than the GEVD solution. An optimization algorithm initialized with the GESD solution is faster than when initialized with the GEVD solution. Median (over $20$ trials) \texttt{cpderr} (left) and computation time (right) for rank $10$ tensors of dimensions $10\times 10\times 10$ with correlated factor matrix columns ($10^\circ$ angles). Initializing the optimization algorithm randomly does not lead to a good solution.}
	\label{optimcorsmall}
\end{figure}
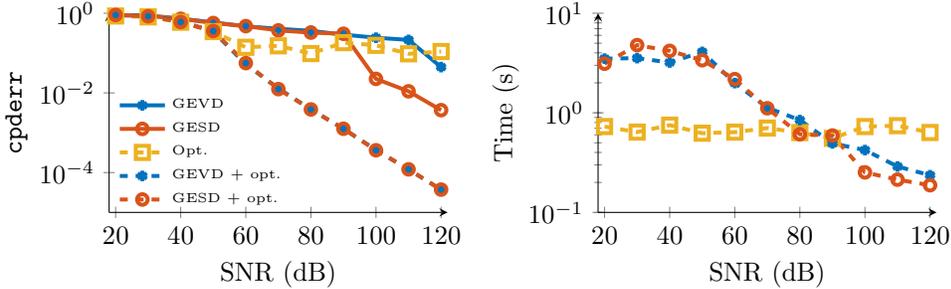

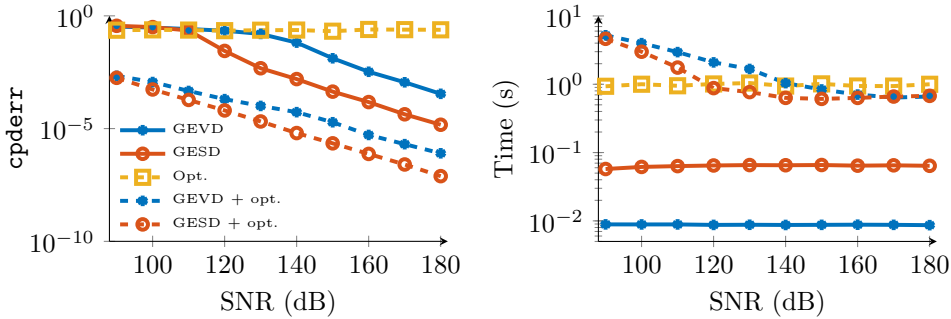
\begin{figure}
	\centering   
%
%
\definecolor{mycolor1}{rgb}{0.00000,0.44700,0.74100}%
\definecolor{mycolor2}{rgb}{0.85000,0.32500,0.09800}%
\definecolor{mycolor3}{rgb}{0.92900,0.69400,0.12500}%
\definecolor{mycolor4}{rgb}{0.49400,0.18400,0.55600}%
\definecolor{mycolor5}{rgb}{0.46600,0.67400,0.18800}%
\begin{tikzpicture}

\begin{axis}[%
	width=4.5cm,
	height=3cm,
	at={(0cm,0cm)},
	scale only axis,
	axis lines = left,
	xmin=88,
	xmax=182,
	xlabel={SNR (dB)},
	ymode=log,
	ymin=1e-10,
	ymax=1,
	yminorticks=true,
	ylabel={\texttt{cpderr}},
	axis background/.style={fill=white},
	legend style={legend cell align=left,align=left,draw=none, at={(0.55,0.57)},fill=none}
]
\addplot [color=mycolor1,solid,line width=1.5pt,mark size=2.0pt,mark=asterisk,mark options={solid}]
table[row sep=crcr]{%
90	0.334463756815703\\
100	0.301935427151377\\
110	0.260075690866929\\
120	0.219838492608813\\
130	0.158176122911358\\
140	0.0651257301397608\\
150	0.0131426061358783\\
160	0.00331584701653166\\
170	0.00113223899040674\\
180	0.000348653537989494\\
};
\addlegendentry{\tiny{GEVD}};

\addplot [color=mycolor2,solid,line width=1.5pt,mark size=2.0pt,mark=o,mark options={solid}]
table[row sep=crcr]{%
90	0.362097837455429\\
100	0.317247410947096\\
110	0.223931126950663\\
120	0.0278600787661532\\
130	0.004748974709824\\
140	0.00157106841421696\\
150	0.000435988483337305\\
160	0.000149984831038957\\
170	4.30833985509058e-05\\
180	1.4831409285225e-05\\
};
\addlegendentry{\tiny{GESD}};

\addplot [color=mycolor3, dashed,line width=1.5pt,mark size=2.5pt,mark=square,mark options={solid}]
table[row sep=crcr]{%
90	0.230891606328338\\
100	0.239118155134309\\
110	0.234138484089438\\
120	0.220608138141708\\
130	0.231113467054405\\
140	0.234683235947475\\
150	0.208737129769673\\
160	0.249030424483022\\
170	0.244516186843899\\
180	0.240827088407389\\
};
\addlegendentry{\tiny{Opt.}};

\addplot [color=mycolor1,dashed,line width=1.5pt,mark size=2.0pt,mark=asterisk,mark options={solid}]
table[row sep=crcr]{%
90	0.00219006329788754\\
100	0.00116112246412343\\
110	0.000466017902223219\\
120	0.000202323770150088\\
130	0.000101257016596229\\
140	5.35580478619866e-05\\
150	1.91409091548852e-05\\
160	5.23043705157663e-06\\
170	2.04162552433609e-06\\
180	8.0299520896475e-07\\
};
\addlegendentry{\tiny{GEVD + opt.}};

\addplot [color=mycolor2,dashed,line width=1.5pt,mark size=2.0pt,mark=o,mark options={solid}]
table[row sep=crcr]{%
90	0.00178667012235756\\
100	0.000544312972625839\\
110	0.000187165327487929\\
120	6.39343632038971e-05\\
130	2.07170789189468e-05\\
140	6.33903189252074e-06\\
150	2.22043729997984e-06\\
160	7.5390312758181e-07\\
170	2.56340840780627e-07\\
180	7.79370441347951e-08\\
};
\addlegendentry{\tiny{GESD + opt.}};

\end{axis}

\begin{axis}[%
	width=4.5cm,
	height=3cm,
	at={(6.5cm,0cm)},
	scale only axis,
	axis lines = left,
	xmin=88,
	xmax=182,
	yticklabel style={
		/pgf/number format/fixed,
		/pgf/number format/precision=2
	},
	xlabel={SNR (dB)},
	ymin=5e-3,
	ymax=10,
	ymode=log,
	ylabel={Time (s)},
	axis background/.style={fill=white},
	legend style={legend cell align=left,align=left,draw=white!15!black}
	]
\addplot [color=mycolor1,solid,line width=1.5pt,mark size=2.0pt,mark=asterisk,mark options={solid}]
  table[row sep=crcr]{%
90	0.00888765\\
100	0.008871\\
110	0.00887495\\
120	0.008732\\
130	0.0087724\\
140	0.00871485\\
150	0.00874835\\
160	0.00879485\\
170	0.00875075\\
180	0.00865855\\
};
\addplot [color=mycolor2,solid,line width=1.5pt,mark size=2.0pt,mark=o,mark options={solid}]
  table[row sep=crcr]{%
90	0.05746255\\
100	0.0616546\\
110	0.0633855\\
120	0.06473235\\
130	0.0656282\\
140	0.0652463\\
150	0.06575985\\
160	0.06433835\\
170	0.0650807\\
180	0.06404345\\
};
\addplot [color=mycolor3, dashed,line width=1.5pt,mark size=2.5pt,mark=square,mark options={solid}]
table[row sep=crcr]{%
90	0.9346572\\
100	1.00118165\\
110	0.95058335\\
120	1.0106505\\
130	1.0424052\\
140	0.94141805\\
150	1.01374\\
160	0.94285675\\
170	0.9407455\\
180	0.9996276\\
};
\addplot [color=mycolor1,dashed,line width=1.5pt,mark size=2.0pt,mark=asterisk,mark options={solid}]
table[row sep=crcr]{%
90	5.214246\\
100	3.9985236\\
110	2.95624495\\
120	2.09116485\\
130	1.6749757\\
140	1.0389998\\
150	0.83926775\\
160	0.70623575\\
170	0.639029\\
180	0.6616489\\
};
\addplot [color=mycolor2,dashed,line width=1.5pt,mark size=2.0pt,mark=o,mark options={solid}]
table[row sep=crcr]{%
90	4.62974315\\
100	2.9966177\\
110	1.75557265\\
120	0.8786383\\
130	0.7681075\\
140	0.6323555\\
150	0.6109882\\
160	0.63163095\\
170	0.66092725\\
180	0.6765798\\
};
\end{axis}
\end{tikzpicture}%
	\caption{ For sufficiently difficult problems, initialization with GESD leads to both a more accurate solution and lower computation time of the optimization algorithm compared to initialization with GEVD. Median (over $100$ trials) \texttt{cpderr} (left) and computation time (right) for rank-$20$ tensors of dimensions $20\times 20\times 20$ with correlated factor matrix columns ($8^\circ$ angles). A splitting threshold of $0.02$ was used in this experiment.}
	\label{optimcorsmall20}
\end{figure}

\subsection{GESD and GEVD vs. noiseless problems}
\label{sec:BadMLSVDSlices}

We now illustrate an additional advantage of GESD compared to GEVD. Namely there is no fixed strategy for pencil selection in GEVD which is able to solve all problems meeting the assumptions of the algorithm even in the noiseless setting. As an example, let $\cT$ be the (real) rank $3$ tensor with frontal slices
\[
\bT_1=\begin{pmatrix} 1 & 0 & 0\\
0 & 1 & 0 \\
0 & 0 & 1
\end{pmatrix} \quad
\bT_2=\begin{pmatrix} 1 & 0 & 0\\
0 & 1/2 & 0 \\
0 & 0 & 1/2
\end{pmatrix} \quad
\bT_3=\begin{pmatrix} 0 & 0 & 0\\
0 & 0 & 1/5 \\
0 & 1/5 & 0
\end{pmatrix}.
\]
As previously discussed, a popular strategy for pencil selection in GEVD is to use the first two slices of a core for the MLSVD of $\cT$. The tensor $\cS$ with frontal slices
\[
\bS_1\approx-\begin{pmatrix} 1.39 & 0 & 0\\
0 & 1.11 & 0 \\
0 & 0 & 1.11
\end{pmatrix} \quad
\bS_2\approx\begin{pmatrix} -0.25 & 0 & 0\\
0 & 0.16 & 0 \\
0 & 0 & 0.16
\end{pmatrix} \quad
\bS_3\approx-\begin{pmatrix} 0 & 0 & 0\\
0 & 0 & 1/5 \\
0 & 1/5 & 0
\end{pmatrix}
\]
is a core for the MLSVD of $\cT$.
If one uses the first two slices of $\cS$ to compute the CPD of $\cT$ via GEVD, then one must compute the eigenvectors of the matrix
\[
\bS_1^{-1} \bS_2 \approx -\begin{pmatrix} 0.181 & 0 & 0\\
0 & 0.143 & 0 \\
0 & 0 & 0.143
\end{pmatrix}.
\]
Since $\spn\{\be_2,\be_3\}$ is an eigenspace of dimension $2$ for this matrix corresponding to the eigenvalue $-0.143$, there is an ambiguity in which eigenvectors should be used. Tensorlab's GEVD implementation, for example, assumes $\be_2$ and $\be_3$ are the correct eigenvectors to take in this subspace and proceeds from there. Instead considering the matrix pencil $(\bS_1,\bS_3)$ shows that $\be_2+\be_3$ and $\be_2-\be_3$ are the correct eigenvectors to take from the subspace $\spn\{\be_2,\be_3\}$.

To make matters worse for GEVD, if one proceeds under the assumption that $\be_1$, $\be_2$, and $\be_3$ are the correct eigenvectors to take, then the GEVD based solution one arrives at is a local optima for the optimization based approach for computing a CPD. Thus using this GEVD based initialization followed by optimization cannot yield the true rank $3$ decomposition of $\cT$.\footnote{The GEVD based decomposition has approximately 0.516 factor matrix error when compared to the correct decomposition.} 

GESD, on the other hand, is easily able to compute the decomposition of this tensor. Although GESD still uses the pencil $(\bS_1,\bS_2)$ on its first iteration, GESD recognizes that the eigenspace $\spn\{\be_2,\be_3\}$ cannot be decomposed with this pencil and delays the decomposition to a later pencil. 

Examined in detail, from the pencil $(\bS_1,\bS_2)$, GESD computes
\[
\bZ_1 = \begin{pmatrix}
1 \\
0 \\ 
0
\end{pmatrix}
\qquad \mathrm{and}
\qquad 
\bZ_2 = \begin{pmatrix}
0 & 0 \\
1 & 0 \\
0 & 1
\end{pmatrix}
\]
and decomposes $\cS = \cS \cdot_2 \bZ_1^\teT + \cS \cdot_2 \bZ_2^\teT$. The tensor $\cS \cdot_2 \bZ_1^\teT$ has rank-$1$ and GESD extracts one rank-$1$ component of $\cT$ from this tensor. $\cS \cdot_2 \bZ_2^\teT$ is a rank-$2$ tensor of size $3 \times 2 \times 3$. Letting $\cA$ denote the core of a $2 \times 2 \times 2$ MLSVD of $\cS \cdot_2 \bZ_2^\teT$, one has
\[
\bA_1 = \begin{pmatrix}
0 & 1.12 \\
1.12 & 0
\end{pmatrix} 
\qquad
\qquad
\bA_2 = \begin{pmatrix}
1/5 & 0 \\
0 & 1/5
\end{pmatrix}.
\]
The generalized eigenvalues of the matrix pencil $(\bA_1,\bA_2)$ are $\spn((1.12,1/5))$ and $\spn((-1.12,1/5))$ with corresponding generalized eigenvectors $\be_1+\be_2$ and $\be_1-\be_2$, respectively. Since this matrix pencil has distinct generalized eigenvalues, GESD successfully extracts the remaining rank-$1$ components of $\cT$ from the pencil.

\section{ Accuracy of GESD versus Corollary \ref{corollary:ComputedJthBound}}
\label{sec:CRboundPerformance} 

In this section we compare the deterministic bound computed in Section \ref{sec:DetermBound} for the existence of two distinct eigenvalue clusters against the accuracy of the first step of our GESD algorithm. We note that the bound computed from the original tensor only provides information about the first step of our algorithm. For this reason, we omit a direct comparison of this bound to the accuracy of the GESD algorithm.

Let $\cT \in \R^{R \times R \times K}$ be a real rank $R$ tensor with CPD $\cpd$, and let $\cN$ be a noise tensor. As previously discussed, given $\cT+\cN$ as an input, the goal of the first step of the GESD algorithm is to robustly decompose $\R^R$ into (at least) two disjoint JGE spaces of $\cT$.

Suppose $\bZ_1$ and $\bZ_2$ are matrices such that the range of each $\bZ_i$ is a JGE space of $\cT$ and such that the matrix $(\bZ_1 \  \bZ_2)$ is invertible. Recall that the $\bZ_i$ have the form $\bZ_i = \bB^{-\teT} \bE_i$ where $\bE_1$ and $\bE_2$ are matrices which for all $j=1,\dots,R$ satisfy that the $j$th row of $\bE_1$ is nonzero if and only if the $j$th row of $\bE_2$ is a zero row. Using this fact we may evaluate the robustness of the first step of our algorithm in two ways. 

Let $\tilde{\bZ}_1$ and $\tilde{\bZ}_2$ be the matrices computed by the first step of our algorithm. Then one should have that the ranges of 
\[
\bB^\teT \tilde{\bZ}_1 =\tilde{\bE}_1 \mathand \bB^\teT \tilde{\bZ}_2 = \tilde{\bE}_2. 
\]
are (approximately) orthogonal. Additionally, there should exist disjoint index sets $J_1,J_2 \subset \{1, \dots, R\}$ such that $J_1 \cup J_2= \{1, \dots, R\}$ and such that
\[
\ran \tilde{\bZ}_1 \approx \mathrm{span}(\{\bf{b}_j\}_{j \in J_1}) \mathand \ran \tilde{\bZ}_2 \approx \mathrm{span}(\{\bf{b}_j\}_{j \in J_2}) 
\]
where the $\bf{b}_j$ are columns of the matrix $(\bB^{-\teT})$. That is, the principal (largest) angle between the range of $\tilde{\bZ}_i$ and $\mathrm{span}(\{\bf{b}_j\}_{j \in J_i})$ should be approximately zero.

In Figure \ref{lowrank}, we plot the median (over 50 trials) of the aforementioned  angles as a function of the SNR of $\cT$ to $\cN$ where $\cN$ is randomly generated Gaussian noise. That is, for many choices of the SNR $s$, we randomly generate noise tensors $\cN$ such that the SNR of $\cT$ to $\cN$ is equal to $s$. In each trial we compute subspaces $\tilde{\bZ}_1$ and $\tilde{\bZ}_2$ for the  tensor $\cT+\cN$ using the GESD algorithm. We check accuracy by plotting the  median over the trials of the angles described above. As before, $\cT$ has been generated by independently sampling all entries of the factor matrices $\bA,\bB,$ and $\bC$ uniformly on $[0,1]$ and forming the associated low-rank tensor. 

The first image of Figure \ref{lowrank} shows the $R \times R \times R$ rank $R$ case with $R=4$, while the case $R=8$ is displayed in the second image. In addition to plotting the aforementioned angles, the  bound from Corollary \ref{corollary:ComputedJthBound} for existence of two distinct eigenvalue clusters is plotted. The  deterministic bound is computed by taking the largest bound found using either five or fifty different unitary matrices and  the subpencil which gives the best bound is the subpencil used to compute generalized eigenspaces in the GESD algorithm.

We find that errors in GESD are well predicted by this bound. For SNR above the deterministic bound, the first step of GESD is highly accurate, while small errors start to appear when SNR decreases below the deterministic bound.
 For all SNR values, the inaccuracy in the orthogonality of the $\tilde{\bE}_i$ and the inaccuracy in the angle between the $\tilde{\bZ}_i$ and the columns of $\bB^{-\teT}$ are observed to have similar magnitude.

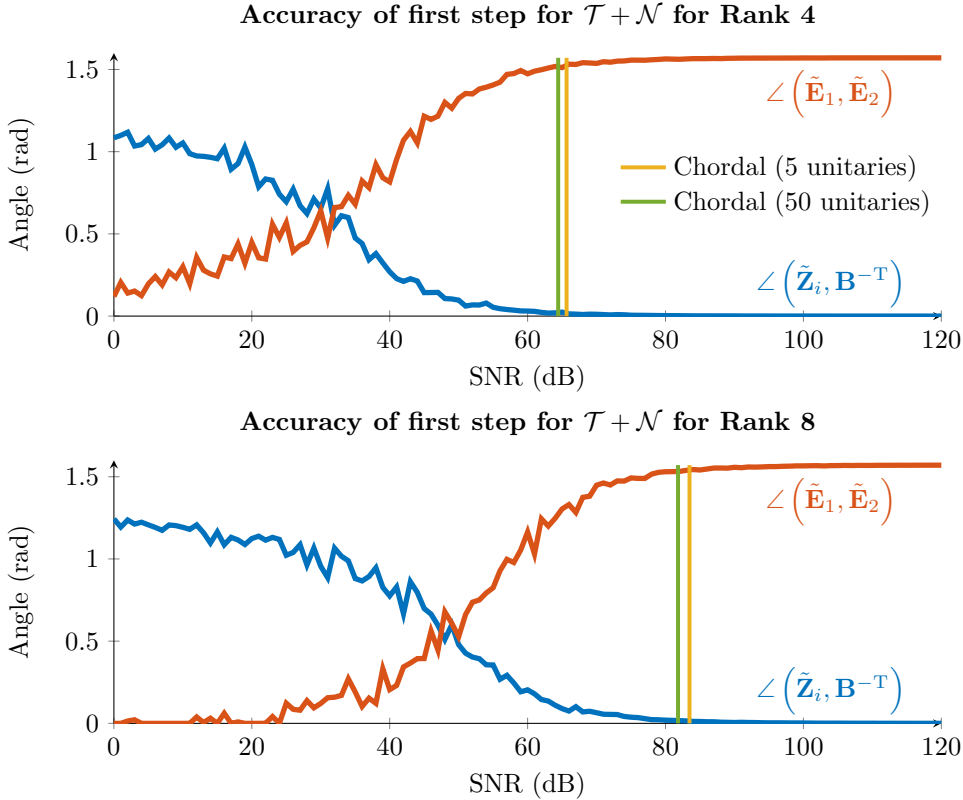
\begin{figure}
	\centering
%
%
\definecolor{mycolor1}{rgb}{0.00000,0.44700,0.74100}%
\definecolor{mycolor2}{rgb}{0.85000,0.32500,0.09800}%
\definecolor{mycolor3}{rgb}{0.92900,0.69400,0.12500}%
\definecolor{mycolor4}{rgb}{0.49400,0.18400,0.55600}%
\definecolor{mycolor5}{rgb}{0.46600,0.67400,0.18800}%
\definecolor{mycolor6}{rgb}{0.30100,0.74500,0.93300}%
\begin{tikzpicture}

\begin{axis}[%
width=11cm,
height=3.5cm,
at={(0cm,0cm)},
scale only axis,
axis lines = left,
xmin=0,
xmax=60,
xticklabels={,0,20,40,60,80,100,120,70,80,90,100,110,120},
xlabel={SNR (dB)},
ymin=0,
ymax=1.6,
ylabel={Angle (rad)},
axis background/.style={fill=white},
title style={font=\bfseries},
title={Accuracy of first step for $\mathcal{T}+\mathcal{N}$ for Rank 4},
axis x line*=bottom,
axis y line*=left,
legend style={at={(0.6,0.5)},anchor=west,legend cell align=left,align=left,draw=white!15!white}
]
\addplot [color=mycolor1,solid,line width=2pt, forget plot]
table[row sep=crcr]{%
0	1.08442228507438\\
0.5	1.0989280634927\\
1	1.1180923667669\\
1.5	1.03448420601397\\
2	1.04285444276875\\
2.5	1.08035892023348\\
3	1.01559512399678\\
3.5	1.04189963491674\\
4	1.08282204896088\\
4.5	1.02669200364507\\
5	1.05288062045649\\
5.5	0.987731946247769\\
6	0.973970155771339\\
6.5	0.971468689091199\\
7	0.965777736548526\\
7.5	0.957024568837467\\
8	1.02116962618742\\
8.5	0.893027397692615\\
9	0.927259165394729\\
9.5	1.0279941352728\\
10	0.921994811375725\\
10.5	0.786377813188594\\
11	0.833747412371551\\
11.5	0.825916313065561\\
12	0.741142263879944\\
12.5	0.695435512515702\\
13	0.771787592035193\\
13.5	0.671365182221602\\
14	0.623673703762665\\
14.5	0.712167429987419\\
15	0.657167426510015\\
15.5	0.774444847204935\\
16	0.552666484064127\\
16.5	0.610501398936938\\
17	0.599094425068754\\
17.5	0.474436671645562\\
18	0.441254366414565\\
18.5	0.343074159675951\\
19	0.379730796154353\\
19.5	0.323176682162234\\
20	0.269391961798447\\
20.5	0.228023299795962\\
21	0.210839963432591\\
21.5	0.227102538552909\\
22	0.213338282913468\\
22.5	0.142764538442288\\
23	0.143307366117571\\
23.5	0.144069861989364\\
24	0.106741517173822\\
24.5	0.104113554981557\\
25	0.0970550044613008\\
25.5	0.0611147682681336\\
26	0.0679608966413554\\
26.5	0.0672774239130103\\
27	0.079035537297486\\
27.5	0.0533260022574917\\
28	0.0450238104824262\\
28.5	0.0398872858518414\\
29	0.0360152460064141\\
29.5	0.03137331159948\\
30	0.0305251088438707\\
30.5	0.0287800113117606\\
31	0.022509004640323\\
31.5	0.0181243599542625\\
32	0.0196689208732269\\
32.5	0.0239766973928377\\
33	0.0133414568708687\\
33.5	0.0119125192139436\\
34	0.00962397283463669\\
34.5	0.0116655753173009\\
35	0.0110030493441646\\
35.5	0.00918040245303095\\
36	0.00704962238192234\\
36.5	0.00913108577205931\\
37	0.00800986793961012\\
37.5	0.00525686984438248\\
38	0.00472320817794856\\
38.5	0.00440728891485704\\
39	0.0038694296517232\\
39.5	0.00329695185135764\\
40	0.00276626705125566\\
40.5	0.00266866802205276\\
41	0.00241895799154614\\
41.5	0.0019454569170192\\
42	0.00162797180137085\\
42.5	0.00138574589515647\\
43	0.00151213725324354\\
43.5	0.00130488775485787\\
44	0.00132937589836936\\
44.5	0.00118305621041191\\
45	0.00111311045134889\\
45.5	0.000782411853540921\\
46	0.000623165463613153\\
46.5	0.000622078785664498\\
47	0.000696876424862005\\
47.5	0.000577073736984597\\
48	0.000379622142530922\\
48.5	0.000368336172310849\\
49	0.000368807810119862\\
49.5	0.000331173295976491\\
50	0.000314626908901474\\
50.5	0.000295458852754496\\
51	0.000188202847743371\\
51.5	0.000271555073377123\\
52	0.000185637160398142\\
52.5	0.000144005266835249\\
53	0.000142535318145243\\
53.5	0.000115242305276654\\
54	0.000101810816427345\\
54.5	0.00012051112719162\\
55	0.000102662815223089\\
55.5	9.79948254991161e-05\\
56	6.19781013222947e-05\\
56.5	5.83887869325378e-05\\
57	6.40949668864603e-05\\
57.5	4.62796312153362e-05\\
58	4.56448543932035e-05\\
58.5	4.59258108746815e-05\\
59	3.75893409713121e-05\\
59.5	3.78775165177411e-05\\
60	3.28936633797382e-05\\
};

\addplot [color=mycolor2,solid,line width=2pt, forget plot]
  table[row sep=crcr]{%
0	0.118952561369691\\
0.5	0.204537147982362\\
1	0.139959171905425\\
1.5	0.152019493784181\\
2	0.124578419736388\\
2.5	0.195916087438876\\
3	0.240446339633316\\
3.5	0.189742276252856\\
4	0.267853251474101\\
4.5	0.233210552871868\\
5	0.300287799277993\\
5.5	0.196713654598208\\
6	0.356504029313773\\
6.5	0.279008058675322\\
7	0.256864400292318\\
7.5	0.2432061926755\\
8	0.359550413664139\\
8.5	0.349693196489885\\
9	0.438784046997763\\
9.5	0.321702757915189\\
10	0.447762084556235\\
10.5	0.354342879712003\\
11	0.346844069037697\\
11.5	0.558440267543486\\
12	0.478272638194809\\
12.5	0.567001523775981\\
13	0.393194666185356\\
13.5	0.429977966201892\\
14	0.451374341152588\\
14.5	0.537077238900148\\
15	0.640255094755577\\
15.5	0.474475517011584\\
16	0.660132144072186\\
16.5	0.665202526215293\\
17	0.734302453898127\\
17.5	0.680714770259186\\
18	0.821888102806634\\
18.5	0.781129773992122\\
19	0.921193521998876\\
19.5	0.841438312648582\\
20	0.816083018320446\\
20.5	0.95270837674327\\
21	1.0694321515374\\
21.5	1.12285929574536\\
22	1.05338177998226\\
22.5	1.21566627794631\\
23	1.19636894241076\\
23.5	1.21611203668914\\
24	1.29956171620282\\
24.5	1.26442958090758\\
25	1.32425722548881\\
25.5	1.35408329046138\\
26	1.35104191323905\\
26.5	1.38125874914929\\
27	1.39203480548583\\
27.5	1.4058567196433\\
28	1.42242849981185\\
28.5	1.47012125608089\\
29	1.47500174702346\\
29.5	1.49350803639744\\
30	1.47364113037655\\
30.5	1.48960483421076\\
31	1.49829815837419\\
31.5	1.50889996217659\\
32	1.51837700593937\\
32.5	1.51134795165603\\
33	1.53214066467798\\
33.5	1.53045989213524\\
34	1.54125618589284\\
34.5	1.53936813607283\\
35	1.53682461911359\\
35.5	1.54713861701157\\
36	1.5441043194813\\
36.5	1.55066512864984\\
37	1.55180776848994\\
37.5	1.55674240809511\\
38	1.55663645651722\\
38.5	1.55617294314795\\
39	1.55895968127536\\
39.5	1.56107877589387\\
40	1.56414768604933\\
40.5	1.5629738084166\\
41	1.56186554274917\\
41.5	1.56406286173942\\
42	1.56580261200227\\
42.5	1.56588125829077\\
43	1.56584596261982\\
43.5	1.56691590859828\\
44	1.56609579069302\\
44.5	1.56773092329632\\
45	1.56825022982041\\
45.5	1.5689955856075\\
46	1.56896491483921\\
46.5	1.56921893294088\\
47	1.56908427346586\\
47.5	1.56940990138051\\
48	1.56924589494505\\
48.5	1.56955427658568\\
49	1.56952049815927\\
49.5	1.56967416917552\\
50	1.57000619243963\\
50.5	1.56992735323552\\
51	1.56994626893463\\
51.5	1.56995274338237\\
52	1.57035755073039\\
52.5	1.57031673077238\\
53	1.57047706832547\\
53.5	1.57043621576937\\
54	1.57052554951571\\
54.5	1.57052360811308\\
55	1.57052917970976\\
55.5	1.5704867326166\\
56	1.5705785605365\\
56.5	1.57059995459526\\
57	1.57056803423777\\
57.5	1.57061624487576\\
58	1.57068269609237\\
58.5	1.57066874736405\\
59	1.5707064617317\\
59.5	1.57068352156424\\
60	1.57070927870917\\
};

\addplot [color=mycolor3,solid,line width=1.5pt]
table[row sep=crcr]{%
	32.8261638083169	0\\
	32.8261638083169	1.5707963267949\\
};
\addlegendentry{Chordal (5 unitaries)};

\addplot [color=mycolor5,solid,line width=1.5pt]
table[row sep=crcr]{%
	32.2174522818788	0\\
	32.2174522818788	1.5707963267949\\
};
\addlegendentry{Chordal (50 unitaries)};

\node [mycolor1] at (520,20){$\angle\left(\tilde{\mathbf{Z}}_i,\mathbf{B}^{-\text{T}}\right)$};
\node [mycolor2] at (520,135){$\angle\left(\tilde{\mathbf{E}}_1,\tilde{\mathbf{E}}_2\right)$};
\end{axis}
\end{tikzpicture}%
%
	\vspace{0.001cm}
%
%
\definecolor{mycolor1}{rgb}{0.00000,0.44700,0.74100}%
\definecolor{mycolor2}{rgb}{0.85000,0.32500,0.09800}%
\definecolor{mycolor3}{rgb}{0.92900,0.69400,0.12500}%
\definecolor{mycolor4}{rgb}{0.49400,0.18400,0.55600}%
\definecolor{mycolor5}{rgb}{0.46600,0.67400,0.18800}%
\definecolor{mycolor6}{rgb}{0.30100,0.74500,0.93300}%
\begin{tikzpicture}

\begin{axis}[%
width=11cm,
height=3.5cm,
at={(0cm,0cm)},
scale only axis,
axis lines = left,
xmin=0,
xmax=60,
xlabel={SNR (dB)},
xticklabels={,0,20,40,60,80,100,120,70,80,90,100,110,120},
ymin=0,
ymax=1.6,
ylabel={Angle (rad)},
axis background/.style={fill=white},
title style={font=\bfseries},
title={Accuracy of first step for $\mathcal{T}+\mathcal{N}$ for Rank 8},
axis x line*=bottom,
axis y line*=left,
legend style={at={(0.78,0.5)},anchor=west,legend cell align=left,align=left,draw=white!15!white}
]
\addplot [color=mycolor1,solid,line width=2pt, forget plot]
  table[row sep=crcr]{%
0	1.24350780509161\\
0.5	1.19396770322793\\
1	1.2368778909416\\
1.5	1.21209619634431\\
2	1.2232805966233\\
2.5	1.20737351069878\\
3	1.19184108819369\\
3.5	1.17396522999557\\
4	1.2066825011671\\
4.5	1.20251770609696\\
5	1.19236655957411\\
5.5	1.17997123126003\\
6	1.20698562462569\\
6.5	1.15923505948619\\
7	1.09710996863001\\
7.5	1.16445698452237\\
8	1.08599926015065\\
8.5	1.13139733863224\\
9	1.11541810820601\\
9.5	1.08774767957783\\
10	1.12407003957571\\
10.5	1.13826106696359\\
11	1.11419886899403\\
11.5	1.13111292155765\\
12	1.11877920478108\\
12.5	1.02235396817541\\
13	1.03955414966936\\
13.5	1.08662944538232\\
14	0.963710897385099\\
14.5	1.0663282842699\\
15	0.955341522683575\\
15.5	0.886446436790791\\
16	1.06566318061716\\
16.5	1.01657478545392\\
17	0.988559977171622\\
17.5	0.880526269261473\\
18	0.866325713229488\\
18.5	0.893162078326091\\
19	0.946398343331817\\
19.5	0.826955145012955\\
20	0.775535416866739\\
20.5	0.829492911047584\\
21	0.666744618006515\\
21.5	0.861516082261816\\
22	0.796877229051756\\
22.5	0.698393626675901\\
23	0.664131734393197\\
23.5	0.596261479673699\\
24	0.51038449100283\\
24.5	0.5995478311603\\
25	0.478619007332422\\
25.5	0.426920132299301\\
26	0.40313240526745\\
26.5	0.393215139448223\\
27	0.356132047955611\\
27.5	0.355176046483987\\
28	0.268106694310575\\
28.5	0.292152828530331\\
29	0.245500242571149\\
29.5	0.19197623989799\\
30	0.203898574894037\\
30.5	0.179473180560566\\
31	0.14494499449862\\
31.5	0.135533388954118\\
32	0.110149784189199\\
32.5	0.0897679466594036\\
33	0.0747464112245304\\
33.5	0.101509182156282\\
34	0.0694795667268616\\
34.5	0.0734317886306115\\
35	0.0646369831105952\\
35.5	0.0554949554130544\\
36	0.0544178015051031\\
36.5	0.0536534409922708\\
37	0.0420692338547676\\
37.5	0.037429059944869\\
38	0.0350444950851431\\
38.5	0.0287721156316967\\
39	0.0225534906791074\\
39.5	0.0218419042117621\\
40	0.0193427857816297\\
40.5	0.0180338302261284\\
41	0.0166392634178517\\
41.5	0.0148838064898852\\
42	0.0119796343036823\\
42.5	0.0113603347622081\\
43	0.00969742736370669\\
43.5	0.00960369262071821\\
44	0.00713130538782562\\
44.5	0.00628083999764332\\
45	0.00515498723679613\\
45.5	0.00645903575002663\\
46	0.00521615338881697\\
46.5	0.00488264925347965\\
47	0.00378181169957044\\
47.5	0.00309974807061839\\
48	0.00325430234785874\\
48.5	0.00302475042309047\\
49	0.00251567782077386\\
49.5	0.0025479031556605\\
50	0.00215009679949692\\
50.5	0.00175699173471805\\
51	0.00134517236187213\\
51.5	0.00119385739082005\\
52	0.00127102851494569\\
52.5	0.00128143114512013\\
53	0.00108150643901727\\
53.5	0.000824131793508544\\
54	0.000781894110346822\\
54.5	0.000705312924460191\\
55	0.000830457329301545\\
55.5	0.000437021815642865\\
56	0.000529758346416541\\
56.5	0.000448175896176946\\
57	0.000408200696329583\\
57.5	0.00036140746455077\\
58	0.000248657267926953\\
58.5	0.00032758887069792\\
59	0.000263545231559142\\
59.5	0.000228697200825278\\
60	0.000188429477966701\\
};

\addplot [color=mycolor2,solid,line width=2pt, forget plot]
  table[row sep=crcr]{%
0	2.98023223876953e-08\\
0.5	0\\
1	0.0144258439334217\\
1.5	0.0217050435654713\\
2	0\\
2.5	0\\
3	1.66600046865626e-08\\
3.5	1.66600046865626e-08\\
4	0\\
4.5	1.66600046865626e-08\\
5	0\\
5.5	0\\
6	0.0387690612155159\\
6.5	0.00525245149091625\\
7	0.0311524281993963\\
7.5	1.97123833825036e-08\\
8	0.061495612638039\\
8.5	0\\
9	0.00885502997021503\\
9.5	1.82501207499443e-08\\
10	0\\
10.5	1.66600046865626e-08\\
11	0.000407971891378824\\
11.5	0.0413141224389494\\
12	1.49011611938477e-08\\
12.5	0.122767689838613\\
13	0.0852224217055857\\
13.5	0.0798993565563472\\
14	0.13826425259879\\
14.5	0.0942683613958098\\
15	0.121093715774899\\
15.5	0.155918763181507\\
16	0.159082902891836\\
16.5	0.169971424665281\\
17	0.270742202137501\\
17.5	0.190746297535148\\
18	0.10257817729441\\
18.5	0.153268729118254\\
19	0.121696944316469\\
19.5	0.331858086770114\\
20	0.205196762146073\\
20.5	0.230543430126397\\
21	0.343175942412575\\
21.5	0.366751397544285\\
22	0.39364582919894\\
22.5	0.394127755162464\\
23	0.561624069081524\\
23.5	0.445025770665638\\
24	0.678261058488536\\
24.5	0.608885006609369\\
25	0.527038190671598\\
25.5	0.66167679670925\\
26	0.736922076583607\\
26.5	0.750736070434876\\
27	0.793092405666413\\
27.5	0.824689577348371\\
28	0.928416623255653\\
28.5	0.997494903590945\\
29	0.978238916666121\\
29.5	1.05416206603715\\
30	1.16388885287401\\
30.5	1.03008963886236\\
31	1.24408151286916\\
31.5	1.19975691323613\\
32	1.24737141638672\\
32.5	1.30410122947447\\
33	1.33053785318343\\
33.5	1.28278612796963\\
34	1.3747462155353\\
34.5	1.38143754689296\\
35	1.44871052982915\\
35.5	1.46183136359775\\
36	1.44969629364839\\
36.5	1.47456932909329\\
37	1.47344066861888\\
37.5	1.49285468191569\\
38	1.48989973801244\\
38.5	1.48964469748668\\
39	1.51607400568327\\
39.5	1.52687902562387\\
40	1.53029206082094\\
40.5	1.53103325671089\\
41	1.53224212808876\\
41.5	1.53872974325674\\
42	1.54376280572337\\
42.5	1.54105139100299\\
43	1.54680817198886\\
43.5	1.5527889868531\\
44	1.5530683460765\\
44.5	1.55226496590452\\
45	1.55731743866626\\
45.5	1.55496860371435\\
46	1.55940319437526\\
46.5	1.55853321716199\\
47	1.55978507481588\\
47.5	1.56164714446033\\
48	1.56166373618129\\
48.5	1.56352961502494\\
49	1.56440293590514\\
49.5	1.566107888111\\
50	1.56599144081222\\
50.5	1.56697376189246\\
51	1.56726915046694\\
51.5	1.56584028243802\\
52	1.56767427412401\\
52.5	1.56790602870893\\
53	1.56886054393465\\
53.5	1.56827892659753\\
54	1.56904087192022\\
54.5	1.56877200314423\\
55	1.56895983042675\\
55.5	1.56964313676171\\
56	1.56950501083253\\
56.5	1.5696477813898\\
57	1.5699460758034\\
57.5	1.56979465113018\\
58	1.57003893826804\\
58.5	1.57016136656742\\
59	1.57019213481575\\
59.5	1.57024167719391\\
60	1.57024586260577\\
};

\addplot [color=mycolor3,solid,line width=1.5pt]
table[row sep=crcr]{%
41.7403139705457	0\\
41.7403139705457	1.5707963267949\\
};

\addplot [color=mycolor5,solid,line width=1.5pt]
table[row sep=crcr]{%
	40.9080322335605	0\\
	40.9080322335605	1.5707963267949\\
};

\node [mycolor1] at (520,20){$\angle\left(\tilde{\mathbf{Z}}_i,\mathbf{B}^{-\text{T}}\right)$};
\node [mycolor2] at (520,135){$\angle\left(\tilde{\mathbf{E}}_1,\tilde{\mathbf{E}}_2\right)$};

\end{axis}
\end{tikzpicture}%
%
	\caption{The angle between $\tilde{\mathbf{E}}_1$ and $\tilde{\mathbf{E}}_2$ stays close to $\pi/2$ until the bound from Corollary \ref{corollary:ComputedJthBound} is crossed, while the angle between $\tilde{\mathbf{Z}}_i$ and columns of $\mathbf{B}^{-\text{T}}$ stays close to zero until this bound is crossed. Trying more unitary matrices shifts the bound sightly to the left. Mean (over $50$ trials) angle between $\tilde{\mathbf{E}}_1$ and $\tilde{\mathbf{E}}_2$ and between $\tilde{\mathbf{Z}}_i$ and columns of $\mathbf{B}^{-\text{T}}$ for a rank-$4$ tensor of dimensions $4\times 4 \times 4$ (top) and for a rank-$8$ tensor of dimensions $8\times 8 \times 8$ (bottom), picking the best unitary matrix out of either $5$ or $50$ trials. The SNR is varied from $0$ dB to $120$ dB. }
	\label{lowrank}
\end{figure}

\section{Conclusion}

We have presented an algebraic generalized eigenspace based algorithm for CPD computation which uses multiple subpencils of a given tensor to compute a decomposition. A key idea in the GESD algorithm is that one may combine information from many subpencils of a tensor to compute its CPD, thus allowing one to only perform trustworthy computations in each pencil. The performance of GESD has been compared to that of the popular GEVD algorithm. We find that in all situations, GESD is at least as accurate as GEVD, and that GESD is typically significantly more accurate than GEVD. We see that GESD particularly shines in difficult problems where the factor matrices of a tensor are highly correlated. 

While the GESD algorithm typically takes about three times as long to run on a core tensor as GEVD, in practical settings which involve compression steps and optimization steps, a CPD computation using GESD for algebraic initialization of optimization is at least as fast as a CPD computation using GEVD for algebraic initialization. We observe that GESD based initializations are often significantly more accurate than GEVD based initializations and that they can lead to more  accurate and reliable optimization. We also illustrate that no single fixed pencil selection strategy for GEVD can solve all given noiseless CPD problems even when the initial tensor meets the assumptions of the GEVD algorithm. Since GESD intelligently determines which computations can be reliably performed with a given subpencil, GESD does not suffer from this issue. 

In addition to experimental analysis, we provide a theoretical analysis of the GESD algorithm. Here we compute a deterministic bound for the existence of a subpencil which has two distinct clusters of generalized eigenvalues which have not coalesced. This bound may be interpreted as a theoretical guarantee that a next step in the GESD algorithm is able to compute generalized eigenspaces which are meaningfully related to the generalized eigenspaces of the original noiseless tensor.

Currently we assume the existence of reasonably sized eigenvalue gaps which may be used to compute at least two distinct generalized eigenspaces. While this assumption is reasonable for many tensors of low-rank, this assumption may fail to hold for particularly ill-conditioned tensors or when rank becomes large. Indeed, for generic tensors, the expected generalized eigenvalue gap becomes arbitrarily small as tensor rank and dimensions increase. Thus for tensors of sufficiently high rank, one should not expect to find a suitable gap. Resolving this is a direction for further research.

\section*{Acknowledgements}

The authors thank Alexander Wagner, Ph.D for discussions which helped improve exposition in the article.  We also thank the anonymous referees for their helpful comments.

\bibliographystyle{siamplain}
\bibliography{references}

\end{document}